\documentclass[12 pt,a4paper]{article}

\usepackage[utf8]{inputenc}
\usepackage[english]{babel}
\usepackage[T1]{fontenc}
\usepackage{graphicx}
\usepackage{lmodern}
\usepackage{amsmath}
\usepackage{tikz}
\usetikzlibrary{patterns,decorations.pathmorphing}
\usepackage{longtable}
\usepackage{epsfig}
\usepackage{subfig}
\usepackage{color}
\usepackage{setspace}
\usepackage[top=2cm,bottom=2.5cm,left=2.5cm,right=2.5cm]{geometry}
\usepackage{wasysym}
\usepackage{hyperref} %[colorlinks=true]
\usepackage[figurewithin=none,tablewithin=none]{caption}
\usepackage{amsmath}
\usepackage{amsfonts}
\usepackage{amssymb}
\usepackage{amsthm}
\usepackage{bm}
\usepackage{dsfont}

\newtheorem{theorem}{Theorem}

\newtheorem{proposition}[theorem]{Proposition}
\newtheorem{remarque}[theorem]{Remark}

\def\R{\mathbb{R}}
\DeclareMathOperator{\sign}{sign}

\title{An explicit pseudo-energy conserving time-integration scheme for Hamiltonian dynamics} 
\author{\begin{minipage}{\textwidth}\centering Frédéric
Marazzato$^{1,2,3}$, Alexandre Ern$^{1,3}$, Christian Mariotti$^{2}$ and Laurent Monasse$^{1,4}$\\
   \small{$^{1}$Universit\'e Paris-Est, Cermics (ENPC), F-77455 Marne-la-Vall\'ee cedex 2, France}\\
   \small{email: \texttt{\{alexandre.ern, frederic.marazzato\}@enpc.fr}}\\
   \small{$^{2}$CEA, DAM, DIF, F-91297 Arpajon, France}\\
   \small{email: \texttt{christian.mariotti@cea.fr}} \\
   \small{$^{3}$Inria Paris, EPC SERENA, F-75589 Paris, France}\\
   \small{$^4$ Inria, Team COFFEE, Sophia Antipolis and Universit\'e
   Nice Sophia Antipolis, CNRS and Laboratoire J. A. Dieudonn\'e, UMR
   7351, 06108 Nice, France}\\
   \small{email: \texttt{laurent.monasse@inria.fr}}\end{minipage}}
   
\begin{document}
\hypersetup{urlcolor=blue,linkcolor=black,citecolor=blue}

\maketitle

\begin{abstract}
We propose a new explicit pseudo-energy and momentum conserving scheme for the time integration of Hamiltonian systems. The scheme, which is formally second-order accurate, is based on two key ideas: the 
integration during the time-steps 
of forces between free-flight particles and the use of 
momentum jumps at the discrete time nodes leading to a two-step formulation
for the acceleration. The pseudo-energy conservation is established under exact force integration, whereas it is valid to second-order accuracy in the presence of quadrature errors.
Moreover, we devise an asynchronous version of the scheme that
can be used in the framework of slow-fast time-stepping strategies. 
The scheme is
validated against classical benchmarks and on nonlinear or 
inhomogeneous wave propagation problems.
\end{abstract}

\section{Introduction}

Energy and momentum conservation is an important property of numerical schemes
for a large number of physical problems. For instance, in statistical
physics, accurately conserving first integrals 
constitutes a fundamental requirement to
capture the correct behaviour of the system. In mechanics, conservation
of the mechanical energy (together with momentum) 
is an important feature for systems such as
the acoustics in a piano~\cite{chabassier2010energy} or nonlinear contact
dynamics~\cite{hauret2006energy, fetecau2003nonsmooth}. In this work, 
we consider Hamiltonian systems consisting of $N$ particles 
in dimension $d$ (typically, $d=1$,
$2$ or $3$) where $q_i,p_i\in\R^d$ are the position and momentum of
the particle $i\in\{1,\dots,N\}$. 
We assume that the Hamiltonian has the following split form:
\begin{equation}
\label{eq:model}
H(\bm{q},\bm{p})
= \frac{1}{2}\bm{p}^{\mathrm{T}} \bm{M}^{-1}\bm{p} + V(\bm{q}),
\end{equation}
where $\bm{q}= (q_1,\dots,q_N)\in\R^{dN}$ is the 
position vector of the particles,
$\bm{p}=
(p_1,\dots,p_N)\in\R^{dN}$ is the momentum vector of the particles, $\bm{M}$ is
the symmetric positive definite mass matrix and $V$ is the potential
energy. The system is thus driven by the equations
\begin{equation}
\dot{\bm{q}} = \bm{M}^{-1}\bm{p},\qquad \dot{\bm{p}} = 
-\nabla V(\bm{q}). \label{eq:hamilton}
\end{equation}

Several approaches have been proposed to tackle the issue of
conservation when integrating numerically (\ref{eq:hamilton}). A first
possibility consists in the use of symplectic
schemes~\cite{hairer2006geometric}, which integrate a modified (not
explicitly exhibited, except in certain simple cases) Hamiltonian and
thus preserve the first integrals of the dynamics over exponentially
long times (with respect to the time-step), up to fluctuations whose
amplitudes grow with the time-step.  However, in the case of variable
time-steps, symplectic schemes lose their conservation properties
since the modified Hamiltonian changes with the
time-step~\cite{calvo1993development}. When the time-step
size is driven by the shape of the Hamiltonian (e.g. in Kepler's
problem with high eccentricity), a workaround consists in adding a
perturbation accounting for the time-step variation in order for a
rescaled dynamic to remain Hamiltonian~\cite{hairer1997variable}. In
practice, for mechanical problems, such a condition on the time-step
can become impractical, since the time-step could be imposed due to
coupling or stiffness phenomena not accounted for in the Hamiltonian
part. For an extended review of variational integrators in mechanics,
we refer the reader to~\cite{marsden2001discrete}.  Another approach consists in
imposing the exact conservation of energy and momentum at each step of
the numerical scheme. Integrating on the constant energy manifold can
be carried out using projection~\cite{hughes1978transient} or Lie
group integration~\cite{iserles2000lie}, but these methods are
computationally expensive as soon as the manifold of constant energy
and momentum has a complex shape. Another class of methods,
energy-momentum conserving schemes, have been proposed
in~\cite{simo1994new, gonzalez1996stability, hauret2006energy,
chabassier2010energy} for nonlinear mechanics, contact mechanics and
nonlinear wave equations, among others. The general principle is to
integrate the nonlinear forces at a special time during the time-step,
which is determined through a nonlinear implicit procedure.  A
higher-order version of these implicit schemes has been derived for
linear wave propagation
in~\cite{chabassier2013introduction}. Variational integrators
combining features of symplectic and energy-momentum schemes have been
developed for variable time-step strategies~\cite{kane1999symplectic}
and nonlinear mechanical problems in~\cite{gross2005conservation}.  

To the best of our knowledge, no explicit pseudo-energy conserving 
scheme has been proposed to date for nonlinear problems. With the 
motivation that explicit schemes often result in greater computational
efficiency, the goal of the present work is to develop such an explicit
scheme for nonlinear mechanics, where pseudo-energy conservation holds exactly for exact force integration and up to second-order accuracy in the presence of quadratures. The present scheme hinges on two key ideas.
The first one, already considered in~\cite{mariotti2016new}, is to 
approximate the dynamics of the particles by free-flight
trajectories during each time-step. The second one is to use
momentum jumps at the discrete time nodes to approximate the 
acceleration. In doing so, we circumvent the
negative result on the existence of explicit schemes
in~\cite[Lemma 3.3]{chabassier2010energy} through the use of a
two-step strategy. This idea has some links with the
implicit energy-conserving average vector field
method \cite{quispel2008new} where the
conservation of the Hamiltonian is formulated
using an implicit integral of the forces
derived from the potential $V$ over the time-step. A high-order
generalization of the average vector field method using collocation
has been developed in~\cite{hairer2010energy}.  The present
numerical scheme shares with average vector field methods the salient
feature of average force integration over each time-step. However, the
two schemes differ on the discretisation of the acceleration, which is
based here on momentum jumps.  

A further development of the present work is to devise an
asynchronous version of our scheme that
lends itself to slow-fast decompositions as presented
in~\cite{hairer2006geometric}, with the goal to further reduce the
computational cost of the simulation. In the case of mechanical systems
with local stiffness, the conditional stability of an explicit
time-integration scheme typically involves small time-steps for the whole
system. A promising direction to mitigate this drawback consists in
using a local time-stepping strategy.  In the linear case, explicit
high-order energy-momentum conserving methods with local time-stepping
have been proposed in~\cite{diaz2009energy}. In the nonlinear
case, a modified St\"ormer--Verlet method for Hamiltonian 
systems containing slow and fast
components is developed in~\cite{hairer2006geometric}. 
It is proved that this time-integrator remains
symplectic, but the ratio of the fast and slow time-steps
strongly influences the error on the total energy and, in general, a
good balance has to be found experimentally. This phenomenon is called
resonance since it is encountered for certain slow/fast
ratios. Similarly, asynchronous variational integrators generally
exhibit resonances when the local time-steps are close to certain
rational ratios, so that ensuring stability requires adequate fitting
of the local time-steps~\cite{fong2008stability}.  In contrast,
the asynchronous version of the present scheme allows 
one to make slow-fast time-integration while conserving a
pseudo-energy (in the absence of quadrature errors). Our numerical
tests show that the asynchronous scheme still exhibits second-order accuracy;
a mathematical proof of this property is postponed to future work.

This paper is organized as follows. In Section \ref{sec:synchronous},
we present the scheme for a Hamiltonian system of interacting particles with a
synchronous time-integration and establish the main properties of the
scheme including second-order accuracy, time-reversibility, linear stability under a CFL condition, and pseudo-energy conservation under exact force integration. In Section \ref{sec:results}, we test 
the synchronous scheme on various benchmarks from the literature including
a nonlinear wave propagation problem. In
Section \ref{sec:asynchronous}, we present the slow-fast time-stepping capabilities of the asynchronous version of the scheme, together with numerical results
on model particle systems connected by springs and on an inhomogeneous wave equation. These results demonstrate the efficiency gains of the asynchronous scheme with
respect to the synchronous scheme. 

\section{Synchronous scheme}\label{sec:synchronous}

In this section, we present our scheme in its synchronous version and establish its main properties.

\subsection{Definition of the scheme}
We consider a sequence of discrete time nodes $t^n$, $n=0,1,\ldots$, 
with time-steps $h_n=t^{n+1}-t^n$ and time intervals $I_n=[t^n,t^{n+1}]$.
The scheme is written at step $n$ as follows: knowing
$\bm{p}^{n-1/2}$, $\bm{q}^n$, and $[\bm{p}]^n$, one computes
\begin{subequations}
\label{schema CM}
\begin{align}
\bm{p}^{n+1/2} &= \bm{p}^{n-1/2} + [\bm{p}]^{n}, \\
\bm{q}^{n+1} &= \bm{q}^n + h_n \bm{M}^{-1}\bm{p}^{n+1/2},\\
\frac{1}{2} \left([\bm{p}]^{n+1} + [\bm{p}]^{n} \right) & =
-\int_{I_n} \nabla V(\bm{\hat{q}}^n(t)) dt, \label{eq:jumps}
\end{align}
\end{subequations}
with the free-flight trajectory over the time interval $I_n$ defined by
\begin{equation} 
\bm{\hat{q}}^n(t) = \bm{q}^n + (t
- t^n)\bm{M}^{-1}\bm{p}^{n+1/2} \quad\forall t\in I_n. 
\label{eq:free_flight}
\end{equation}
Here, $[\bm{p}]^n$ represents the jump of the momentum vector at time $t^n$,
$\bm{q}^n$ the position vector at time $t^n$, and
$\bm{p}^{n+1/2}$ is the momentum vector between $t^n$ and $t^{n+1}$. 
We observe that $\bm{q}^{n+1} = \bm{\hat{q}}^n(t^{n+1})$. We
initialize the scheme as follows:
\begin{equation}
\label{eq:init_scheme}
\bm{p}^{-1/2} = \bm{p}(t^0), \quad \bm{q}^0
= \bm{q}(t^0), \quad [\bm{p}]^{0}=\bm{0},
\end{equation}
where $\bm{q}(t^0),\bm{p}(t^0)$ are the given position and momentum vectors at the initial time $t^0$.
The scheme (\ref{schema CM}) can alternatively be written as the
following 2-step scheme without jumps: knowing $\bm{p}^{n-1/2}$, $\bm{q}^n$,
and $\bm{p}^{n+1/2}$, one computes
\begin{subequations}
\label{schema multistep}
\begin{align}
\bm{q}^{n+1} &= \bm{q}^n + h_n \bm{M}^{-1}\bm{p}^{n+1/2},\\
\frac{1}{2} \left(\bm{p}^{n+3/2} - \bm{p}^{n-1/2} \right) & =
-\int_{I_n} \nabla V(\bm{\hat{q}}^n(t)) dt, \label{eq:jumps2}
\end{align}
\end{subequations}
with the free-flight trajectory defined by~\eqref{eq:free_flight}.
The initialization of the scheme, equivalent to~\eqref{eq:init_scheme}, is as follows:
\begin{equation}
\label{eq:init_scheme_bis}
\bm{p}^{-1/2} = \bm{p}(t^0), \quad \bm{q}^0
= \bm{q}(t^0), \quad \bm{p}^{1/2} = \bm{p}(t^0).
\end{equation}
This initialization is tailored to achieve exact pseudo-energy conservation under exact force integration, as shown in Theorem \ref{energy conservation} below. Other choices for the initialization are possible, for instance using a one-step method.

In the numerical implementation of the scheme, the integral
in (\ref{eq:jumps}) (or in \eqref{eq:jumps2}) 
is usually not computed exactly but with a
quadrature of the form
\begin{equation}
Q_n(f(t);t^n;t^{n+1}) = h_n\sum_{i=0}^{I} \omega_i f(\lambda_i t^n + (1-\lambda_i)t^{n+1})
\approx \int_{I_n} f(t) dt,
\end{equation}
where the real numbers $\omega_i$ are the weights and the real numbers $\lambda_i\in [0,1]$ define the quadrature points.
Applying the quadrature componentwise for the calculation of the forces and exploiting that the position of the particles varies linearly in time during the free flight, we obtain
\begin{equation}
Q_n\left(\nabla V(\bm{\hat{q}}^n(t));t^n;t^{n+1}\right) = h_n\sum_{i=0}^{I} \omega_i \nabla V(\lambda_i \bm{q}^n + (1-\lambda_i) \bm{q}^{n+1}))
\approx \int_{I_n} \nabla V(\bm{\hat{q}}^n(t)) dt,
\label{eq:quadrature}
\end{equation}
and we replace (\ref{eq:jumps}) with
\begin{equation}
\frac{1}{2} \left([\bm{p}]^{n+1} + [\bm{p}]^{n} \right)  = -Q_n\left(\nabla V(\bm{\hat{q}}^n(t));t^n;t^{n+1}\right), \label{eq:jumps_inexact}
\end{equation}
and a similar modification for \eqref{eq:jumps2}.
In what follows, we assume that the quadrature is symmetric:
\begin{equation}
\forall i\in\{0,\dots,I\},\quad \omega_i = \omega_{I-i} \quad \mathrm{ and
} \quad \lambda_i = 1-\lambda_{I-i}, \label{eq:symmetry}
\end{equation}
and at least of order two (i.e., that the quadrature integrates 
exactly affine polynomials). We also assume that $V$ is of class $C^2$,
i.e., $V \in C^2(\mathbb{R}^{dN};\R)$. 
This implies that
\begin{equation}
Q_n\left(\nabla V(\bm{\hat{q}}^n(t));t^n;t^{n+1}\right) = \int_{I_n}{\nabla
V(\bm{\hat{q}}^n(t)) dt} + \mathcal{O}(h_n^3).\label{eq:order2}
\end{equation}

\subsection{Properties of the scheme}

We now establish various properties of the scheme: pseudo-energy conservation (in the absence of quadrature errors), symmetry (or time-reversibility), second-order accuracy, and linear stability (with constant time-step).

\begin{theorem} [Pseudo-energy conservation]
\label{energy conservation}
Assume that the quadrature is exact.
Then, the scheme (\ref{schema CM}) exactly conserves the
following pseudo-energy:
\begin{equation}
\label{discrete energy}
\tilde{H}^n := V(\bm{q}^n) + \frac{1}{2}\left(\bm{p}^{n-1/2}\right)^{\mathrm{T}} \bm{M}^{-1} \bm{p}^{n+1/2}.
\end{equation}
Moreover, denoting $H^0:=H(\bm{q}(t^0),\bm{p}(t^0))$ the value
of the exact Hamiltonian at the initial time time $t^0$, we have $\tilde{H}^n = H^0$ for all $n\geq 0$ if the scheme is initialized using \eqref{eq:init_scheme}. 
\end{theorem}

\begin{proof}
Using (\ref{schema CM}.a), \eqref{eq:free_flight}, and the chain
rule, we obtain
\[\frac{d}{dt} \left( V(\bm{\hat{q}}^n(t)) \right) = \nabla
V\left(\bm{\hat{q}}^n(t) \right) \cdot (\bm{\hat{q}}^n)'(t)
= \left(\bm{M}^{-1} \bm{p}^{n+1/2}\right)^{\mathrm{T}} \nabla V\left(\bm{\hat{q}}^n(t) \right) . \]
Integrating in time and using (\ref{eq:jumps}) and the symmetry of $\bm{M}$, we infer that
\begin{align*}
 V(\bm{q}^{n+1}) - V(\bm{q}^n) &= \left(\bm{p}^{n+1/2}\right)^{\mathrm{T}} \bm{M}^{-1} \int_{t^n}^{t^{n+1}} \nabla V(\bm{\hat{q}}^n(t)) dt  \\
 &= -\left(\bm{p}^{n+1/2}\right)^{\mathrm{T}} \bm{M}^{-1} \frac{1}{2} \left( [\bm{p}]^{n+1} + [\bm{p}]^{n} \right) \\ &= -\left(\bm{p}^{n+1/2}\right)^{\mathrm{T}} \bm{M}^{-1} \frac{1}{2} \left( \bm{p}^{n+3/2} - \bm{p}^{n-1/2} \right).
\end{align*}
This leads to 
\[ V(\bm{q}^{n+1})
+ \frac{1}{2}\left(\bm{p}^{n+1/2}\right)^{\mathrm{T}} \bm{M}^{-1} \bm{p}^{n+3/2} =
V(\bm{q}^n)
+ \frac{1}{2}\left(\bm{p}^{n-1/2}\right)^{\mathrm{T}} \bm{M}^{-1} \bm{p}^{n+1/2},\] 
showing that $\tilde{H}^{n+1}=\tilde{H}^n$, thereby proving the first assertion. 
Finally, using the initialization \eqref{eq:init_scheme}, we obtain
$\tilde{H}^0=H^0$, and this concludes the proof.
\end{proof}

\begin{remarque} [Quadratures]
\label{rq:quadratures}
In practice, the integral in Equation (\ref{eq:jumps}) can be computed exactly only for polynomial potentials $V$. 
For instance, using the $n$-point Gauss--Lobatto (resp., Gauss--Legendre) quadrature, polynomials of degree up to $2n-3$ (resp., $2n-1$) are integrated exactly. 
The use of quadratures instead of exact integration for general nonlinear potentials entails only an approximate conservation of the pseudo-energy. Since the scheme is second-order accurate (see Theorem \ref{order of the method} below), we expect that pseudo-energy conservation is second-order accurate at best: 
\begin{equation}
\label{eq:pseudo-energy conservation with quadratures}
\tilde{H}^n = \tilde{H}^0 + O(h^2).
\end{equation}
where $h := \sup_n h_n$.
The order of the quadrature has an influence on the multiplicative constant in $O(h^2)$, with the constant (swiftly) decreasing when increasing the quadrature order. Numerical results are presented on a nonlinear wave propagation problem in Section~\ref{sec:nonlinear wave equation}.
\end{remarque}

%\begin{remarque} [\cor{Non-symplecticity}]
%\cor{The proposed method is not symplectic because it exactly preserves the energy of the system. This feature is forbidden to symplectic methods because of a famous theorem cited in \cite{marsden1988lie}. }
%\TODO{Méthode probablement pas symplectique mais regarder. Peut se prouver ? Ou bien test numérique ?}
%\end{remarque}

\begin{remarque}[Momentum conservation]
Let $(1,...,1)$ be the vector of size $dN$ filled with ones. Assume that the system is isolated, i.e., $(1,..,1)^{\mathrm{T}} \cdot \nabla V(\bm{q}) = \bm{0}$
for all $\bm{q}\in \R^{dN}$. Then, the total momentum, defined as $P_{n+1/2}:= (1,..,1)^{\mathrm{T}} \cdot \bm{p}_{n+1/2}$ for all $n\ge0$, is conserved. This follows by 
taking the product of $(1,...,1)^{\mathrm{T}}$ with Equation (\ref{eq:jumps2}) and using the null initialisation of the momentum jump which follows from Equation (\ref{eq:init_scheme}).
\end{remarque}

\begin{proposition} [Symmetry]
\label{sym and reverse}
If the quadrature
(\ref{eq:quadrature}) is exact or symmetric, 
then the scheme (\ref{schema multistep})  is symmetric (or time-reversible).
\end{proposition}

\begin{proof}  
Let $\bm{Y}^n = (
        \bm{q}^n,
        \frac{\bm{p}^{n-1/2}+\bm{p}^{n+1/2}}{2},
        \bm{p}^{n+1/2}-\bm{p}^{n-1/2})^{\mathrm{T}}$. Since we are going to consider positive and negative time-steps in this proof, we denote by $\sign(h_n)$ the sign of the time-step. 
The numerical scheme can be written as $\bm{Y}^{n+1} = \bm{\Phi}_{h_n}(\bm{Y}^n)$, where for a generic column vector 
$\bm{Y} = ( \bm{Y}_1,\bm{Y}_2,\bm{Y}_3)^{\mathrm{T}}$, we have
\[
\bm{\Phi}_{h_n}\left(\bm{Y}\right) = \left( \begin{gathered}
\bm{Y}_1 + h_n\bm{M}^{-1}\left(\bm{Y}_2+\frac{\sign(h_n)}{2}\bm{Y}_3\right) \\
\bm{Y}_2 - Q_n\left(\nabla
V\left(\bm{Y}_1+t\bm{M}^{-1}\left(\bm{Y}_2+\frac{\sign(h_n)}{2}\bm{Y}_3\right)\right)
; 0 ; h_n\right) \\
-\bm{Y}_3 - 2\sign(h_n)Q_n\left(\nabla
V\left(\bm{Y}_1+t\bm{M}^{-1}\left(\bm{Y}_2+\frac{\sign(h_n)}{2}\bm{Y}_3\right)\right)
; 0 ; h_n\right)
\end{gathered}
\right), 
\]
where we used the invariance by translation of the quadrature $Q_n$.
Therefore, we have
\[\bm{\Phi}_{-h_n}\left(\bm{Y}\right) = 
      \left( \begin{gathered}
\bm{Y}_1 - h_n\bm{M}^{-1}\left(\bm{Y}_2-\frac{\sign(h_n)}{2}\bm{Y}_3\right) \\
\bm{Y}_2 - Q_n\left(\nabla
V\left(\bm{Y}_1+t\bm{M}^{-1}\left(\bm{Y}_2-\frac{\sign(h_n)}{2}\bm{Y}_3\right)\right)
; 0 ; -h_n\right) \\
-\bm{Y}_3 + 2\sign(h_n)Q_n\left(\nabla
V\left(\bm{Y}_1+t\bm{M}^{-1}\left(\bm{Y}_2-\frac{\sign(h_n)}{2}\bm{Y}_3\right)\right)
; 0 ; -h_n\right)
\end{gathered}
\right).\]
It remains to verify that $\bm{\Phi}_{h_n} \circ \bm{\Phi}_{-h_n}\left(\bm{Y}\right)  = \bm{Y}$ using that the quadrature is symmetric or exact. To fix the ideas, we assume that $h_n>0$. Let us write $\bm{Y}'=\bm{\Phi}_{-h_n}(\bm{Y})$ and $\bm{Y}''=\bm{\Phi}_{h_n}(\bm{Y}')$. Since $\bm{Y}'_1=\bm{Y}_1 - h_n\bm{M}^{-1}(\bm{Y}_2-\frac{1}{2}\bm{Y}_3)$ and $\bm{Y}'_2+\frac12 \bm{Y}'_3=\bm{Y}_2-\frac12 \bm{Y}_3$, we infer that
\begin{equation}\label{eq:reversibility1}
\bm{Y}''_1 = \bm{Y}'_1 + h_n\bm{M}^{-1}\left(\bm{Y}'_2+\frac{1}{2}\bm{Y}'_3\right) = \bm{Y}_1.
\end{equation}
For the second component, we obtain
\[
\begin{aligned}
\bm{Y}''_2 &= \bm{Y}'_2 - Q_n\left(\nabla
V\left(\bm{Y}_1+(t-h_n)\bm{M}^{-1}\left(\bm{Y}_2-\frac{1}{2}\bm{Y}_3\right)\right)
; 0 ; h_n\right)\\
&= \bm{Y}'_2 - Q_n\left(\nabla
V\left(\bm{Y}_1+t\bm{M}^{-1}\left(\bm{Y}_2-\frac{1}{2}\bm{Y}_3\right)\right)
; -h_n ; 0\right)\\
&= \bm{Y}'_2 + Q_n\left(\nabla
V\left(\bm{Y}_1+t\bm{M}^{-1}\left(\bm{Y}_2-\frac{1}{2}\bm{Y}_3\right)\right)
; 0 ; -h_n\right) = \bm{Y}_2,
\end{aligned}
\]
where we used (\ref{eq:reversibility1}) in the first line and
invariance by translation and symmetry of the quadrature $Q_n$ in the
second and third lines respectively. The proof that 
$\bm{Y}''_3 = \bm{Y}_3$ follows along similar lines.
\end{proof}

\begin{theorem} [Consistency error]
\label{order of the method}
Assume that $V \in C^2(\mathbb{R}^{dN};\R)$. If the quadrature
(\ref{eq:quadrature}) is exact or at least of order two, the scheme~\eqref{schema CM} has second-order accuracy in time.
\end{theorem}

\begin{proof} 
Let $\bm{q}(t),\bm{p}(t)$ be the exact solution to (\ref{eq:hamilton}).
Let us consider the column vector 
$\bm{Y}(t^n) = (\bm{q}(t^n),\frac{\bm{p}(t^{n-1/2})+\bm{p}(t^{n+1/2})}{2},\bm{p}(t^{n+1/2})-\bm{p}(t^{n-1/2}))^{\mathrm{T}}$. The consistency error is defined as
\[\bm{\eta}^{n+1}
:= \frac{\bm{Y}(t^{n+1})-\bm{\Phi}_{h_n}(\bm{Y}(t^{n}))}{h_n},
\]
where $\bm{\Phi}_{h_n}$ is defined in the previous proof. Let us denote by
$\bm{\eta}_1^{n+1},\bm{\eta}_2^{n+1},\bm{\eta}_3^{n+1}$ the three components of the consistency error. 
We have
\begin{align*}
h_n \bm{\eta}_1^{n+1} &= \bm{q}(t^{n+1})-\bm{q}(t^n)-h_n\bm{M}^{-1}\bm{p}(t^{n+1/2}),\\
h_n \bm{\eta}_2^{n+1} &= \frac{\bm{p}(t^{n+3/2})-\bm{p}(t^{n-1/2})}{2}+Q_n\left(\nabla
V(\bm{\check{q}}^n(t));t^n;t^{n+1}\right),\\
h_n \bm{\eta}_3^{n+1} &= 2h_n\bm{\eta}_2^{n+1},
\end{align*}
where $\bm{\check{q}}^n(t) = \bm{q}(t^n) + \bm{M}^{-1}\bm{p}(t^{n+1/2})(t-t^n)$.
Using a Taylor expansion and the equation $\dot{\bm{q}}(t) = \bm{M}^{-1}\bm{p}(t)$, we infer that 
\[
h_n \bm{\eta}_1^{n+1} =
h_n\dot{\bm{q}}(t^{n+1/2})-h_n\bm{M}^{-1}\bm{p}(t^{n+1/2})+\mathcal{O}(h_n^3)
= \mathcal{O}(h_n^3).
\]
Moreover, since the quadrature is of second-order (at least) so that it can be replaced by the mid-point quadrature up to $\mathcal{O}(h_n^3)$, and using the equations $\dot{\bm{q}}(t) = \bm{M}^{-1}\bm{p}(t)$ and $\dot{\bm{p}}(t) = -\nabla V(\bm{q})(t)$, we obtain
\begin{align*}
h_n {\bm{\eta}}_2^{n+1} &= h_n\dot{\bm{p}}(t^{n+1/2})-h_n\nabla
V({\bm{\check{q}}}(t^{n+1/2}))+\mathcal{O}(h_n^3) \\
&= h_n\dot{\bm{p}}(t^{n+1/2})-h_n\nabla
V(\bm{q}(t^{n})+ \tfrac12 h_n \dot{\bm{q}}(t^{n+1/2}))+\mathcal{O}(h_n^3) \\
&= h_n\dot{\bm{p}}(t^{n+1/2})-h_n\nabla
V(\bm{q}(t^{n+1/2}))+\mathcal{O}(h_n^3) = \mathcal{O}(h_n^3).
\end{align*}
We conclude that $\bm{\eta}^{n+1}=\mathcal{O}(h_n^2)$,
i.e., the scheme is second-order accurate.
\end{proof}

\begin{proposition} [Linear stability] 
\label{absolute stability}
Assume that the potential $V$ is quadratic with a positive definite 
Hessian $\bm{H}:=D^2V$. Let $\lambda$ be
the largest eigenvalue of $\bm{H}$. Let $\mu>0$ be the
smallest eigenvalue of $\bm{M}$.  Then the scheme \eqref{schema CM} 
is conditionally stable for a constant time-step $h$
under the following CFL condition:
\begin{equation}\label{eq:abs_stab}
h \ < \ 2\sqrt{\frac{\mu}{\lambda}}.
\end{equation}
\end{proposition}

\begin{proof}
Since the potential $V$ is quadratic, the dynamical system \eqref{eq:hamilton}
is linear. Let $\bm{Z}^n$ be the column vector such that $\bm{Z}^n = \left(
\bm{q}^n,\bm{p}^{n-1/2},\bm{p}^{n+1/2}\right)^{\mathrm{T}}$. Adding a linear functional to $V$ does not change the nature of the Hamiltonian system. We thus consider $\nabla V(\bm{0}) = \bm{0}$ and $V(\bm{0}) = 0$. Since $\nabla V(\bm{q})$ is by assumption linear in $\bm{q}$, we have $\nabla V(\bm{\hat{q}}^n(t)) = \nabla V(\bm{q}^n) + (t-t^n)\bm{H}\bm{M}^{-1} \bm{p}^{n+1/2}$, so that
\[
\int_{I_n} \nabla V(\bm{\hat{q}}^n(t)) dt = h \bm{H}\bm{q}^n + \frac12 h^2 \bm{H}\bm{M}^{-1} \bm{p}^{n+1/2}.
\]
Therefore, the scheme (\ref{schema CM}) can be written as $\bm{Z}^{n+1} = \bm{A} \bm{Z}^n$ with 
\[\bm{A} = \left(
      \begin{array}{ccc}
        \bm{I}_{dN} & \bm{0}_{dN} & h\bm{M}^{-1} \\
        \bm{0}_{dN} & \bm{0}_{dN} & \bm{I}_{dN} \\
        -2 h \bm{H} & \bm{I}_{dN} & -h^2\bm{H}\bm{M}^{-1}
      \end{array}
\right).
\] 
The matrix
$\bm{M}$ being symmetric definite positive, its square root $\bm{M}^{1/2}$ 
is well-defined. We then observe that
\begin{align*}
\tilde{\bm{A}} &= \left(\begin{array}{ccc}\bm{M}^{1/2} & \bm{0}_{dN}
& \bm{0}_{dN} \\ \bm{0}_{dN} & \bm{M}^{-1/2}
& \bm{0}_{dN} \\ \bm{0}_{dN} & \bm{0}_{dN}
& \bm{M}^{-1/2}\end{array}\right) \bm{A} \left(\begin{array}{ccc}\bm{M}^{-1/2} & \bm{0}_{dN}
& \bm{0}_{dN} \\ \bm{0}_{dN} & \bm{M}^{1/2}
& \bm{0}_{dN} \\ \bm{0}_{dN} & \bm{0}_{dN}
& \bm{M}^{1/2}\end{array}\right) \\
&= \left(
      \begin{array}{ccc}
        \bm{I}_{dN} & \bm{0}_{dN} & h\bm{I}_{dN} \\
        \bm{0}_{dN} & \bm{0}_{dN} & \bm{I}_{dN} \\
        -2 h \bm{S} & \bm{I}_{dN} & -h^2\bm{S}
      \end{array}
\right),
\end{align*}
where we introduced the symmetric positive definite matrix $\bm{S} =
\bm{M}^{-1/2}\bm{H}\bm{M}^{-1/2}$. Up 
to an adequate change of variable for each of the
coordinates, it is possible to assume that $\bm{S}$ is
diagonal. 
Denoting $(\sigma_i)_{1\leq i\leq dN}$ the eigenvalues of
$\bm{S}$ and scaling the momenta in $\bm{Z}_n$ by the factors 
$(\sqrt{2\sigma_i})_{1\leq i\leq dN}$, $-\tilde{\bm{A}}$ is block diagonal in
the following matrices of order 3, for all $i\in\{1,\dots,dN\}$:
%\begin{equation*}
%a_{n,i} = \left(
%      \begin{array}{ccc}
%        1 & 0 & h_n\sqrt{2\sigma_i} \\
%        0 & 0 & 1 \\
%        -h_n \sqrt{2\sigma_i} & 1 & -h_n^2\sigma_i
%      \end{array}
%\right). 
%\end{equation*}

\begin{equation*}
a_{i} = \left(
      \begin{array}{ccc}
        -1 & 0 & -h\sqrt{2\sigma_i} \\
        0 & 0 & -1 \\
        h \sqrt{2\sigma_i} & -1 & h^2\sigma_i
      \end{array}
\right).
\end{equation*}
%Preuve simple
The characteristic polynomial $\chi_{a_i}$ of $a_i$ is 
$\chi_{a_i}(X) = (X-1) (X^2 - X (h^2 \sigma_i - 2) + 1 )$, 
which shows that $1$ is an eigenvalue of $a_i$. Moreover,
the polynomial $X^2 - X (h^2 \sigma_i - 2) + 1$ is positive as long as
$h \  <  \ \frac{2}{\sqrt{\sigma_i}}$, and the two complex conjugate
eigenvalues, written $b_i$ and $\overline{b}_i$, have a modulus equal
to $1$ and a nonzero imaginary part. Thus, the three eigenvalues are distinct, and the $3{\times}3$ matrix $a_i$ can be diagonalized for all $i\in\{1,\dots,dN\}$. 
Let $P_i$ the matrix such that $P_i^{-1} a_i P_i = \mathrm{Diag}(1,b_i,\overline{b}_i)$.
Then, writing $P^{-1}$ the block-diagonal matrix composed of the elementary matrices $(P^{-1}_i)_{1\le i\le dN}$ and $D$ the block-diagonal matrix composed of the diagonal matrices of eigenvalues of $(a_i)_{1\le i\le dN}$, we infer that, for all $n \in \mathbb{N}$:
\[\Vert \tilde{\bm{A}}^n \Vert \leq \Vert P^{-1} \Vert \Vert P \Vert \Vert D \Vert^n \leq \Vert P^{-1} \Vert \Vert P \Vert,  \]
%\[\Vert \tilde{\bm{A}}_n \Vert \leq \Vert Q \Vert \Vert D \Vert \Vert P \Vert \leq \Vert Q \Vert \Vert P \Vert  \]
%\[\Vert \tilde{\bm{A}} \Vert \leq \Vert Q \Vert \Vert P \Vert \Vert D \Vert \leq \Vert Q \Vert \Vert P \Vert  \]
because the diagonal matrix $D$ has diagonal entries of modulus 1, and thus $\Vert D \Vert = 1$. Hence, for all $n \in \mathbb{N}$, we obtain $\Vert \bm{Z}^n \Vert  \leq C \Vert \bm{Z}^0 \Vert$, for a constant $C$ independent of $n$.
%$\vert \bm{Z}^n \vert  \leq \Vert P^{-1} \Vert \Vert P \Vert \Vert M^{1/2} \Vert^2 \vert \bm{Z}^{0} \vert$.
Since the eigenvalues $\sigma_i$ of $\bm{S}$ are positive and smaller than
$\frac{\lambda}{\mu}$, we conclude that linear stability holds true under the CFL condition \eqref{eq:abs_stab}.

\end{proof}

\begin{remarque} [Comparison with St\"ormer--Verlet]
A possible writing of the St\"ormer--Verlet method is the following:
\begin{equation*}
\begin{aligned}
\bm{q}^{n+1} &= \bm{q}^n + h_n \bm{M}^{-1}\bm{p}^{n+1/2},\\
\bm{p}^{n+3/2} &= \bm{p}^{n+1/2} - h_{n+1} \nabla V(\bm{q}^{n+1}). \\
\end{aligned}
\end{equation*}
Both the present scheme and the St\"ormer--Verlet scheme 
are of leapfrog-type, 
have a similar CFL condition for linear stability, and are 
second-order accurate.
The main difference is that, using a mid-point quadrature, the forces used to update the momenta at $t^{n+1}$ are computed at $t^{n+1/2}$ with the present scheme, i.e., $\frac12(\bm{p}^{n+3/2}-\bm{p}^{n-1/2})=-h_n\nabla V(\bm{\hat{q}}(t^{n+1/2}))$, whereas the momentum update can be rewritten as
\[\frac12(\bm{p}^{n+3/2}-\bm{p}^{n-1/2})=-\frac12\Big(h_{n+1} \nabla V(\bm{q}^{n+1})+
h_{n} \nabla V(\bm{q}^{n})\Big),\] 
in the St\"ormer--Verlet scheme. Moreover,
the origin of energy conservation is different for the two schemes. The St\"ormer--Verlet scheme is energy-conserving only for constant time-steps due to its symplecticity. The present scheme enjoys an algebraic pseudo-energy preservation property for every time-step (constant or not) up to quadrature errors.
\end{remarque}

\begin{remarque}[Adaptive time-stepping for discrete energy control]
\label{rq:a posteriori}
The conservation of the pseudo-energy does not imply
stability since
$\left(\bm{p}^{n-1/2}\right)^{\mathrm{T}} \bm{M}^{-1} \bm{p}^{n+1/2}$
does not have a sign a priori. However, 
%denoting \ale{$\bm{p}^n:= \frac12(\bm{p}^{n-1/2}+ \bm{p}^{n+1/2})$}, and 
defining the discrete energy 
\begin{equation}
\label{eq:discrete energy bis}
H^n := V(\bm{q}^n) + \frac{1}{8} \left( \bm{p}^{n-1/2}
+ \bm{p}^{n+1/2} \right)^{\mathrm{T}} \bm{M}^{-1} \left( \bm{p}^{n-1/2}
+ \bm{p}^{n+1/2} \right),
\end{equation} 
a straightforward calculation shows that
\begin{equation} \label{eq:identity_H}
H^n = \tilde{H}^n + \frac{1}{8}([\bm{p}]^n)^{\mathrm{T}} M^{-1}
[\bm{p}]^n = \tilde{H}^n + \frac{1}{8} \vert M^{-1/2}
[\bm{p}]^n \vert^2.
\end{equation}
This implies that 
\[0\leq H^n-\tilde{H}^n\leq \mathcal{O}(h_n^2),\]
where we used the identity \eqref{eq:identity_H} for the lower bound and we invoked Theorem \ref{order of the method} for the upper bound. One can use the identity \eqref{eq:identity_H} during the computations for an on-the-fly monitoring of possible departures of the conserved pseudo-energy $\tilde{H}^n$ from the discrete energy $H^n$. The idea is to check whether $\frac{1}{8} \vert M^{-1/2} [\bm{p}]^n \vert^2 \le \epsilon_{\textrm{fly}} \tilde{H}^n$ after every momenta computation and to halve the time-step if this bound is not met (note that the momentum jumps converge to zero with the time-step). The benefits of such an adaptive time-stepping strategy are illustrated in Section \ref{sec:nonlinear wave equation}.
\end{remarque}

% \begin{remarque} [\cor{Nonlinear CFL condition}]
% \cor{The CFL condition of Equation (\ref{eq:abs_stab}) is strictly valid only in the linear case. The search for a similar result for nonlinear potentials is still ongoing. However, verifying the a posteriori criterion of Equation (\ref{eq:identity_H}) at every iteration requires little effort of implementation and ensures the energetic stability of the method.}
% \end{remarque}

\section{Numerical results}
\label{sec:results}

In this section, we present numerical results for the scheme \eqref{schema CM}. We consider classical benchmarks from the literature and a nonlinear wave equation from~\cite{chabassier2010energy}. 

\subsection{Convergence study}
\label{sec:convergence study one particle}
We perform a convergence study with a single particle in
dimension $d=1$.
The reference solution is $q(t) = \sin(t)^4 + 1$, and 
the corresponding potential energy is
\[V(q(t)) = 8 \left( (q(t)-1)^{3/2} - (q(t)-1)^2 \right). \]
We apply the scheme \eqref{schema CM} to this Hamiltonian system 
over $10^3$ seconds using the mid-point quadrature 
as well as the three- and five-point Gauss--Lobatto 
quadratures of order 3 and 7, respectively, 
for the integration of the forces.
We report the $\ell_1$-error with respect to the reference solution
(the sum of the errors at the discrete time nodes
divided by the number of time-steps) 
in Figure \ref{l1-convergence method} as a function of the
number of force evaluations. We observe that for the three
quadratures, the convergence is of
second order as expected. The quadrature order does not impact the
convergence rate but has an influence on the computational
efficiency. We note that in this case, the mid-point quadrature
is more efficient than the three- and five-point Gauss--Lobatto quadratures of order 3 and 7.
 % \cor{A comparison has been made between the initialisation of Equation (\ref{eq:init_scheme_bis}) with an initialization using a one-step method for the momenta. $\bm{p}^{-1/2} := \bm{p}(t^0)$ is kept. However, $\bm{p}^{1/2} := \bm{p}^{-1/2} + h_0 \nabla V(\bm{q}^0)$ is chosen for the concurrent initialisation.
% The resulting values are presented in Figures \ref{l1-convergence method diff initialisation} and \ref{l1-convergence method}. No major difference in accuracy can be seen between the two initialisation strategies. }

\begin{figure}[!htp]
\begin{center}
\includegraphics[scale=0.5]{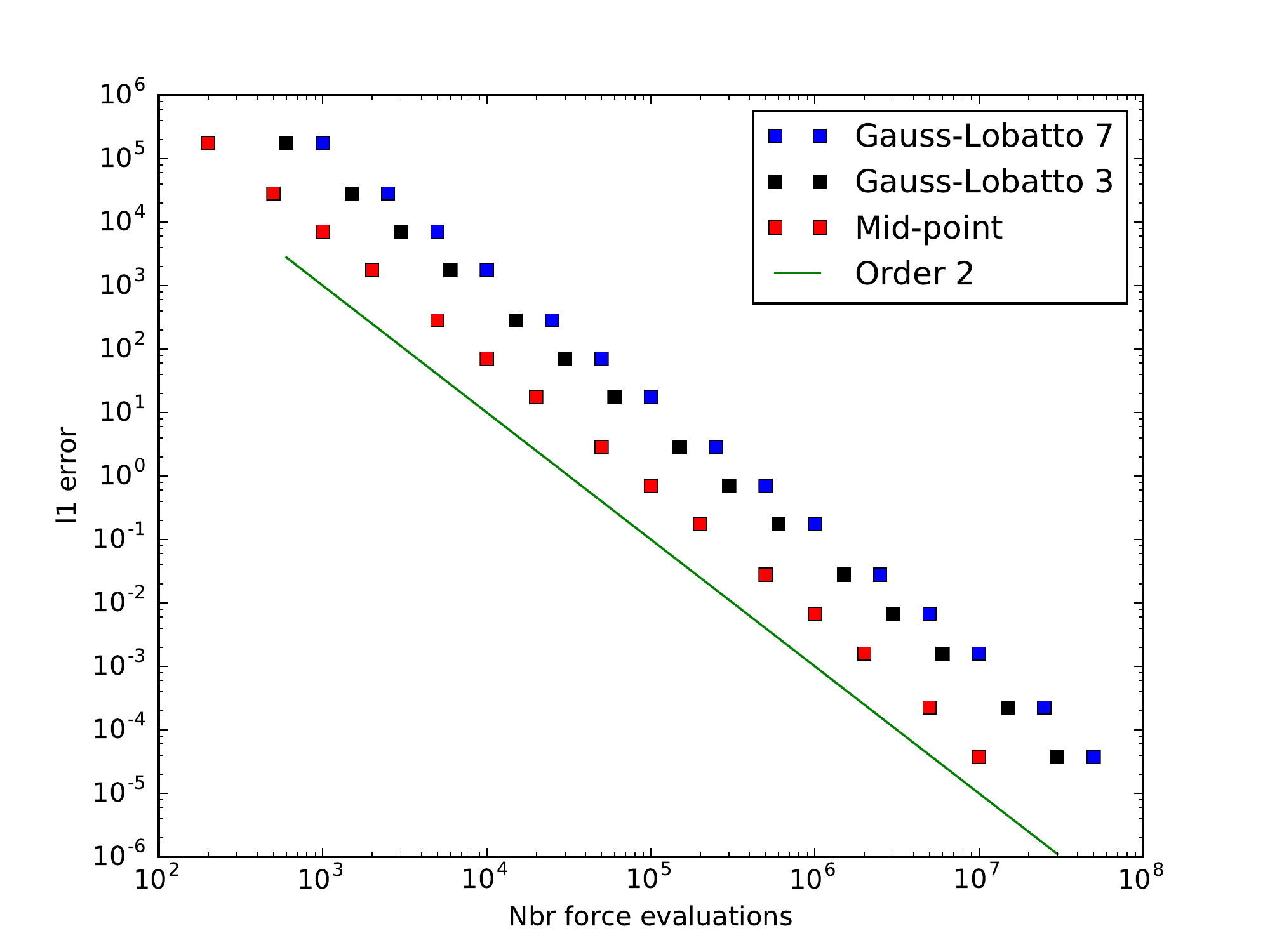}
\caption{Convergence test: $\ell_1$-convergence for a single particle}
\label{l1-convergence method}
\end{center}
\end{figure}

% \begin{figure}[!htp]
% \begin{center}
% \subfloat[]{\label{l1-convergence method}
% \resizebox{0.5\textwidth}{!}{
% \includegraphics{convergence_l1.pdf}}
% }
% \subfloat[]{\label{l1-convergence method diff initialisation}
% \resizebox{0.5\textwidth}{!}{
% %\input{convergence_l1_comparaison}
% \includegraphics{convergence_l1_comparaison.pdf}}%[trim = 1cm 1cm 0.5cm 0.25cm, clip]
% }
% \caption{Convergence test: \protect\subref{l1-convergence method} $\ell_1$-convergence for a single particle with Equation (\ref{eq:init_scheme})}
% %\label{energy variation fpu}
% \end{center}
% \end{figure}

\subsection{Fermi--Pasta--Ulam}
This test case was proposed
in~\cite[Chap. I.4]{hairer2006geometric}. It consists in having stiff
linear springs linked to soft nonlinear springs in an alternating
way, in dimension $d=1$. Figure \ref{schema_fermi} illustrates the setting.
\begin{figure}[!htp]
\begin{center}
\begin{tikzpicture}
\draw[fill=black!10!white, draw=none] (-0.7,-1.5) rectangle (0.,1.5);
\draw[pattern=north east lines, pattern color=black, draw=none] (-0.7,-1.5) rectangle (0.,1.5);
\draw[thick] (0,-1.5) -- (0,1.5);
\def\coilfine#1{
        {#1 +  (1.5-0.4)/7*\t - 0.2*(cos(\t * pi r)-1)},
        {0.4 * sin(\t * pi r)}
        }
\def\coilthick#1{
        {#1 +  (1.5-0.2)/15*\t - 0.1*(cos(\t * pi r)-1)},
        {-0.4 * sin(\t * pi r)}
        }
\draw[domain={0:7},smooth,variable=\t,samples=100] plot (\coilfine{0});
\draw[domain={0:15},smooth,variable=\t,samples=100, thick] plot (\coilthick{1.5});
\draw[domain={0:7},smooth,variable=\t,samples=100] plot (\coilfine{3});
\draw[domain={0:15},smooth,variable=\t,samples=100, thick] plot (\coilthick{4.5});
\draw[domain={0:7},smooth,variable=\t,samples=100] plot (\coilfine{6});
\draw[domain={0:15},smooth,variable=\t,samples=100, thick] plot (\coilthick{7.5});
\draw[domain={0:7},smooth,variable=\t,samples=100] plot (\coilfine{9});
\draw[fill=black!10!white, draw=none] (10.5,-1.5) rectangle (11.2,1.5);
\draw[pattern=north east lines, pattern color=black, draw=none] (10.5,-1.5) rectangle (11.2,1.5);
\draw[thick] (10.5,-1.5) -- (10.5,1.5);
\draw[fill] (0,0) circle [radius=3pt];
\draw[fill] (1.5,0) circle [radius=3pt];
\draw[fill] (3,0) circle [radius=3pt];
\draw[fill] (4.5,0) circle [radius=3pt];
\draw[fill] (6,0) circle [radius=3pt];
\draw[fill] (7.5,0) circle [radius=3pt];
\draw[fill] (9,0) circle [radius=3pt];
\draw[fill] (10.5,0) circle [radius=3pt];
\draw (1.5,0.4) node[anchor=south]{$q_1$};
\draw (3,0.4) node[anchor=south]{$q_2$};
\draw (5.25,0.4) node[anchor=south]{$\cdots$};
\draw (7.5,0.4) node[anchor=south]{$q_{2m-1}$};
\draw (9,0.4) node[anchor=south]{$q_{2m}$};
\draw (2.25,-0.4) node[anchor=north]{\begin{minipage}{2cm}\centering
        stiff harmonic\end{minipage}};
\draw (6.75,-0.4) node[anchor=north]{\begin{minipage}{2cm}\centering
        soft nonlinear\end{minipage}};
\end{tikzpicture}
\end{center}
\caption{Fermi--Pasta--Ulam test case}
\label{schema_fermi}
\end{figure}
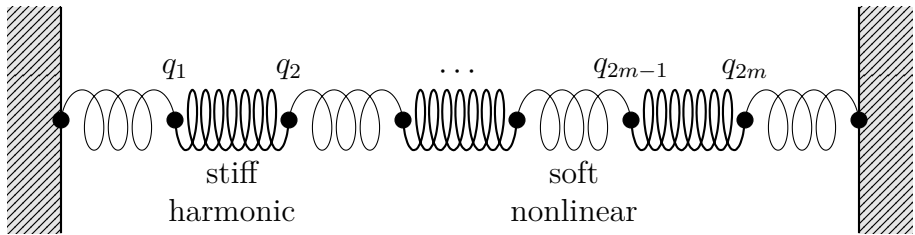

The Hamiltonian is \[H(\bm{q},\bm{p}) = \frac{1}{2} \sum_{i=1}^m (p^2_{2i-1} + p^2_{2i}) + \frac{\omega^2}{4} \sum_{i=1}^m (q_{2i} - q_{2i-1})^2 + \sum_{i=0}^m (q_{2i+1} - q_{2i})^4, \]
with typically $\omega\gg1$.
Introducing the variables
\begin{alignat*}{2}
x_i &= (q_{2i} + q_{2i-1})/ \sqrt{2}, &\qquad y_i &= (p_{2i} + p_{2i-1})/ \sqrt{2}, \\
x_{m+i} &= (q_{2i} - q_{2i-1})/ \sqrt{2}, &\qquad y_{m+i} &= (p_{2i} - p_{2i-1})/ \sqrt{2}, \\
\end{alignat*}
the Hamiltonian can be rewritten as
\begin{multline}
H(\bm{x},\bm{y}) = \frac{1}{2} \sum_{i=1}^{2m} y_i^2 + \frac{\omega^2}{2} \sum_{i=1}^m x_{m+i}^2 \\ + \frac{1}{4} \left( (x_1 - x_{m+1})^4 + \sum_{i=1}^{m-1} (x_{i+1} - x_{m+i+1} - x_i - x_{m+i})^4 + (x_m + x_{2m})^4 \right).
\end{multline}
As the system is Hamiltonian, the total energy of the system should be
conserved by the numerical scheme. The Fermi--Pasta--Ulam system has yet
another quasi-invariant. Letting $I_j(x_{m+j},y_{m+j})
= \frac{1}{2} \left( y^2_{m+j} + \omega^2 x^2_{m+j} \right)$ be the
oscillatory energy of the $j$th stiff spring, the total oscillatory energy $I =
I_1 + I_2 + \dots + I_m$ is close to a constant value as proved
in \cite[p.22]{hairer2006geometric}: 
\[ I(\bm{x}(t),\bm{y}(t)) =
I(\bm{x}(0),\bm{y}(0)) + \mathcal{O}(\omega^{-1}).\]

In our numerical experiment, we set $m=3$ and $\omega =
50$. Figure \ref{energy variation fpu} (left panel) shows the variation of the
oscillating energies and of the pesudo-energy $\tilde{H}^n$ over
time for a constant time-step $h=10^{-3}$. 
The energy exchange between the oscillatory modes
is remarkably similar to the reference solution given
in \cite[Chap. I.4]{hairer2006geometric} and represented in
Figure \ref{fpu reference} (right panel). The reference solution was computed with high accuracy using a Runge--Kutta 4 integrator with a time-step of
$h = 10^{-4}$.
In particular, the total
oscillatory energy $I$ displays fast oscillations around a fixed
constant. The conservation of energy is verified up to machine
precision, even with a mid-point quadrature. The results being already very satisfactory, the results computed with higher order quadratures are omitted for brevity. A more detailed study of the influence of the order of quadrature on pseudo-energy conservation is presented in the next section.

\begin{figure}[!htp]
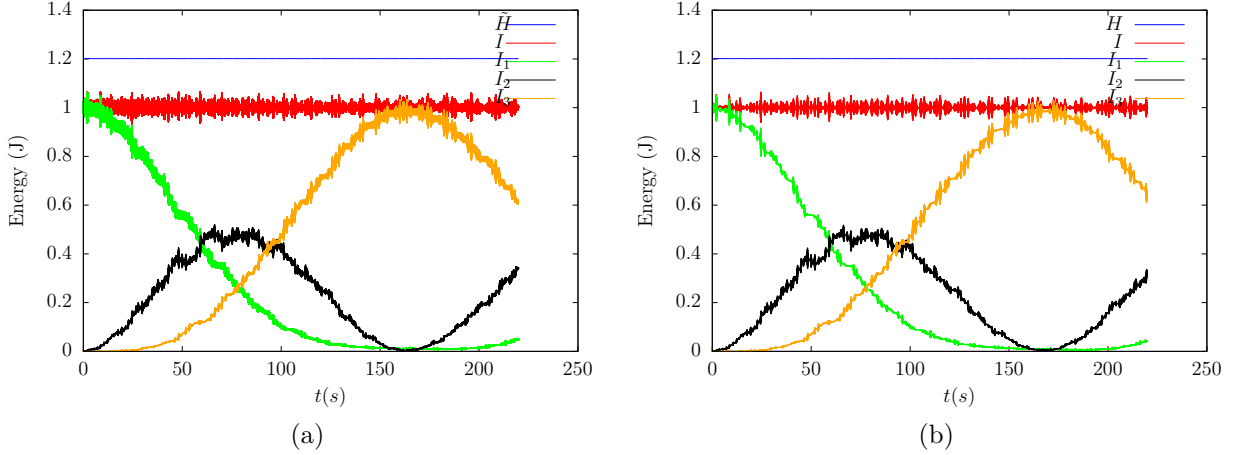

\begin{center}
\subfloat[]{\label{energy variation fpu}
\resizebox{0.5\textwidth}{!}{\input{fermi}}
}
\subfloat[]{\label{fpu reference}
\resizebox{0.5\textwidth}{!}{\input{FPU}}
}
\caption{Fermi--Pasta--Ulam test case: \protect\subref{energy
variation fpu} Energy variation, present scheme, $h=10^{-3}$
; \protect\subref{fpu reference} reference RK4 solution, $h=10^{-4}$}
%\label{energy variation fpu}
\end{center}
\end{figure}

\subsection{Nonlinear wave equation}
\label{sec:nonlinear wave equation}

The setting comes from~\cite{chabassier2010energy}. The interval
$\Omega = [0,1]$ represents a one-dimensional string. Let
$V:\R^d\to\R$ be the potential energy, with dimension
$d=2$. It is assumed that $V$ verifies the following
conditions:
\begin{itemize}
\item
Smoothness: $V$ is of class $C^2$;
\item
Convexity: $V$ is strictly convex;
\item
Coercivity: $\exists K>0$ so that $V(u) \geq K \vert u \vert^2$ for all $u\in \R^2$;
\item Boundedness:
$\exists M>0$ so that $\vert \nabla V(u) \vert^2 \leq M \min(V(u),(1+ \vert u \vert^2))$ for all $u\in \R^2$.
\end{itemize}
The problem of interest is to find $u: \Omega \times \mathbb{R}^+ \rightarrow \mathbb{R}^2$ such that 
\begin{equation}
\left\{
\begin{aligned}
&\partial^2_{tt} u - \partial_x( \nabla V(\partial_x u)) = 0, \\
&u(0,t) = 0, \quad u(1,t) = 0, \\
&u(x,0) = u^0(x),\quad\partial_t u(x,0) = v^0(x),
\end{aligned}
\right.
\end{equation}
with given initial conditions $u^0 : \Omega \rightarrow \R^2$ and $v^0 : \Omega \rightarrow \R^2$.
For a pair $(u_1,u_2) \in \mathbb{R}^2$, the 
functional $V$ takes the following value: 
\[ V(u_1,u_2)
= \frac{u_1^2+u_2^2}{2} - \alpha \left( \sqrt{(1+u_1)^2 + u_2^2} -
(1+u_1) \right),
\] 
where the parameter $\alpha \in [0,1)$ is related to the
tension of the string, such that the string behaviour is nonlinear when
$\alpha>0$ and the strength of the nonlinearity increases with $\alpha$.
The following variational formulation in $\mathcal{V} := H^1_0(\Omega;\R^2)$ is considered: 
\[\frac{d^2}{d t^2} \left( \int_{\Omega} \bm{u} \cdot \bm{v} \right) + \int_{\Omega} \nabla V( \partial_x \bm{u}) \cdot \partial_x \bm{v} = 0, \quad \forall \bm{v} \in \mathcal{V}, \; \forall t >0. 
\]

We use $H^1$-conforming $\mathbb{P}_1$ Lagrange finite elements
for the space discretization. 
Let $N$ be the number of nodes discretizing the string and
$\left(\varphi_i\right)_{1\le i\le 2N}$ be the nodal basis functions associated 
with the degrees of freedom of the string in the two directions.
These basis functions span the finite-dimensional subspace $\mathcal{V}_N \subsetneq \mathcal{V}$. 
The space semi-discrete function approximating the exact solution is
$\bm{u}_N(t) = \sum_{i=1}^{2N} \bm{q}_i(t) \varphi_i(x) \in \mathcal{V}_N$ and solves the following space semi-discrete problem:
\[\frac{d^2}{d t^2} \left( \int_{\Omega} \bm{u}_N \cdot \bm{v}_N \right) + \int_{\Omega} \nabla V( \partial_x \bm{u}_N) \cdot \partial_x \bm{v}_N = 0, \quad \forall \bm{v}_N \in \mathcal{V}_N, \;\forall t >0. 
\]
Introducing the vector $\bm{q}=(q_1,\ldots,q_{2N}) \in \R^{2N}$,
the following Hamiltonian system has to be integrated in time:
\[H(\bm{q},\bm{p}) = \frac{1}{2}\bm{p}^{\mathrm{T}} \bm{M}^{-1}\bm{p} + V(\bm{q}), \qquad
V(\bm{q}) = \int_{\Omega} V\left(\sum_{i=1}^{2N} \bm{q}_i \partial_x \varphi_i \right), \]
where $\bm{M}$ is the classical $\mathbb{P}_1$ Lagrange finite element mass matrix. Assuming that all the components of $\bm{q}$ associated with the first direction are enumerated first and then those associated with the second direction, the matrix $\bm{M}$ is a $2\times 2$ block-diagonal matrix and each diagonal block is a tridiagonal matrix of size $N\times N$ equal to $\Delta x \mathop{\mathrm{tridiag}}(1/6,2/3,1/6)$.

In our numerical experiments, we consider the values
$\alpha = 0$ (which corresponds to the linear case), $\alpha = 0.8$
(which corresponds to a mildly nonlinear behavior), and $\alpha = 0.99$ (which
corresponds to a strongly nonlinear behavior). The space discretisation is such that $\Delta x = 0.01$ and thus $N=99$ basis functions are used in each direction. The time-step is $\Delta t = 0.0033$. Using the same space discretisation, the greatest stable constant time-step has been found to be $\Delta t_{\text{max}} = 0.0055$.
Three numerical simulations are performed in every case 
by letting the amplitude of the 
initial condition $u^0$
at time $t^0$ be $0.01$, $0.1$, or $0.3$. The initial velocity
at time $t^0$ is always taken to be zero. 
The results are reported in Figure \ref{defs} where in all cases,
a mid-point quadrature is used. Six snapshots of the $\R^2$-valued deformation vector $u_0(x) + u(x,t)$ of the string over one second are represented horizontally in various colors; specifically, at each snapshot in time, the deformation vector is plotted in the corresponding vertical plane.
%Figure \ref{defs} thus presents the evolution of the vector-valued displacements $u(x,t)$ over one second.}
The role played by the
nonlinearity can be observed in the fact that the amplitude of $u^0$
influences the vibration of the string. The tension which causes
nonlinearity also changes the wave celerity. We observe an excellent
agreement between the present results and the results reported
in~\cite{chabassier2010energy}. 

\begin{figure}[htp!]
\centering
\subfloat{
\includegraphics[width=0.33\textwidth]{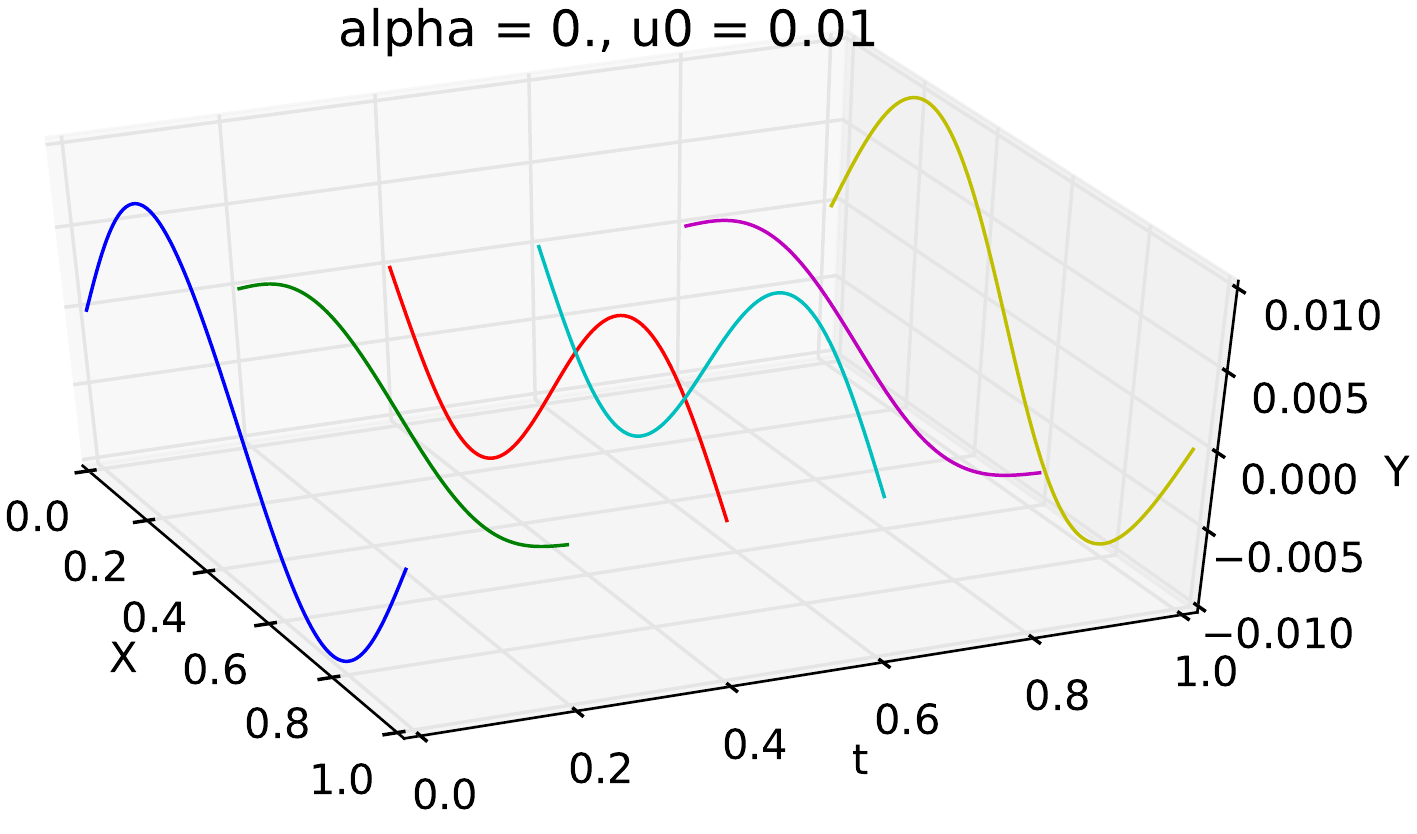}
}
\subfloat{
\includegraphics[width=0.33\textwidth]{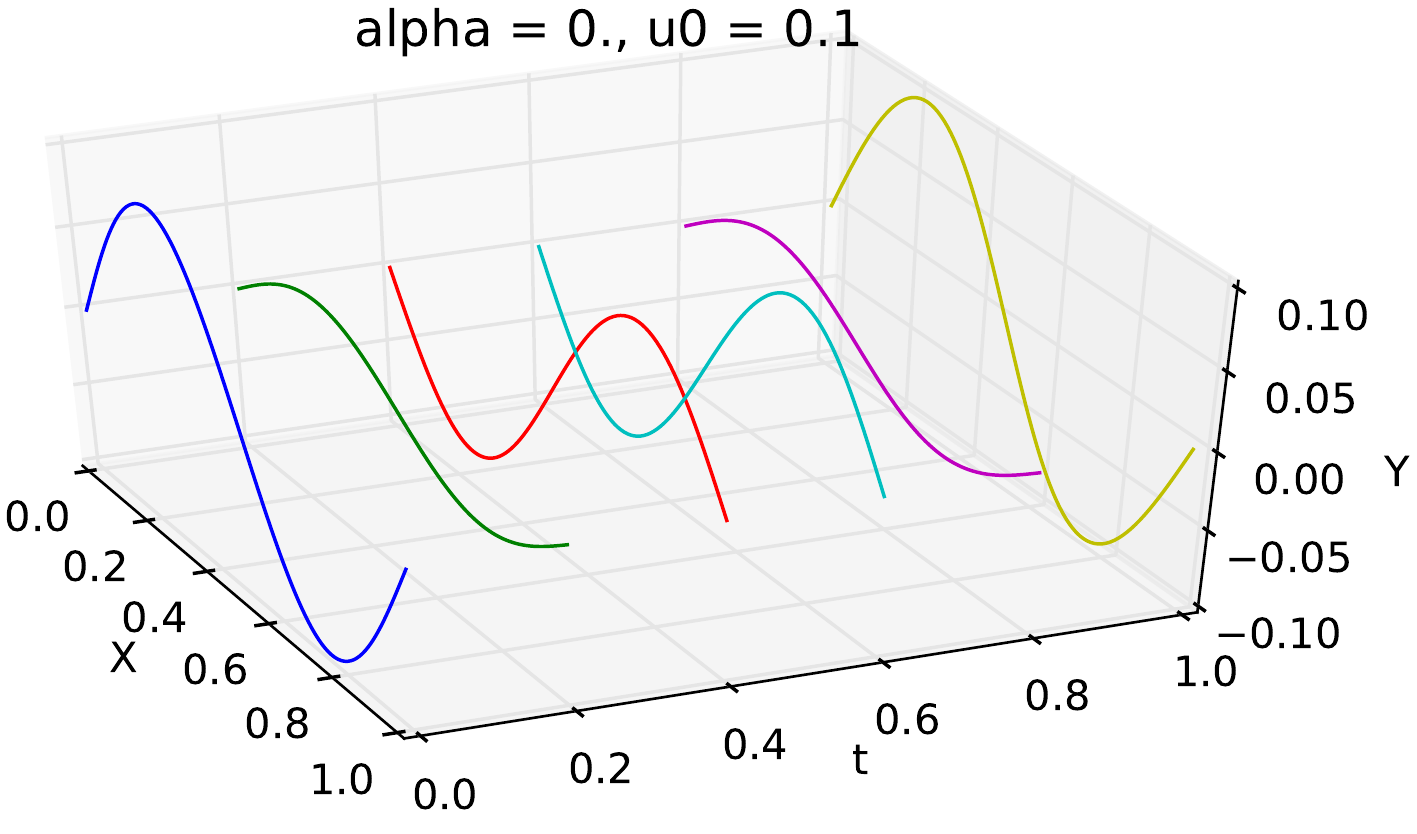}
}
\subfloat{
\includegraphics[width=0.33\textwidth]{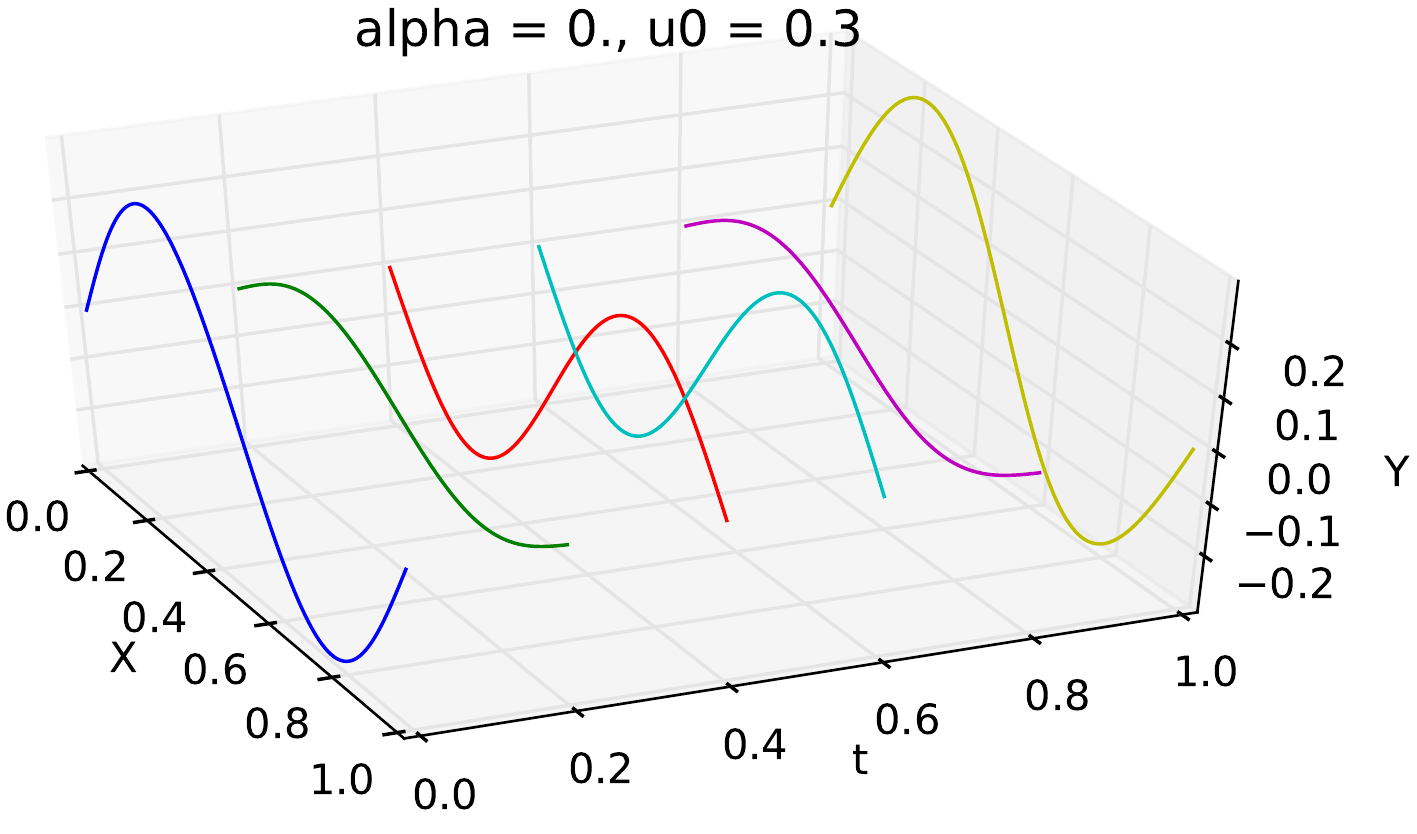}
}\\
\subfloat{
\includegraphics[width=0.33\textwidth]{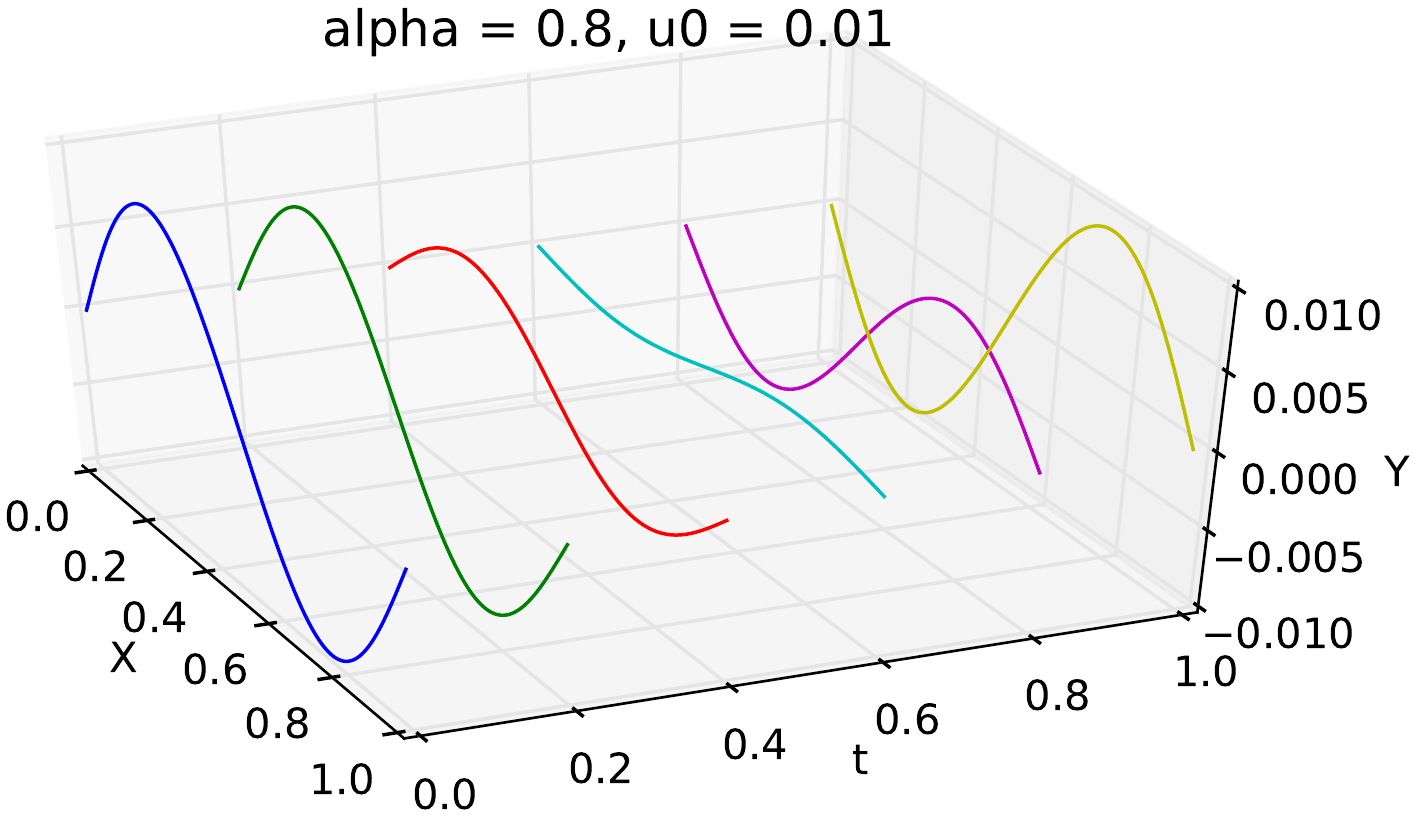}
}
\subfloat{
\includegraphics[width=0.33\textwidth]{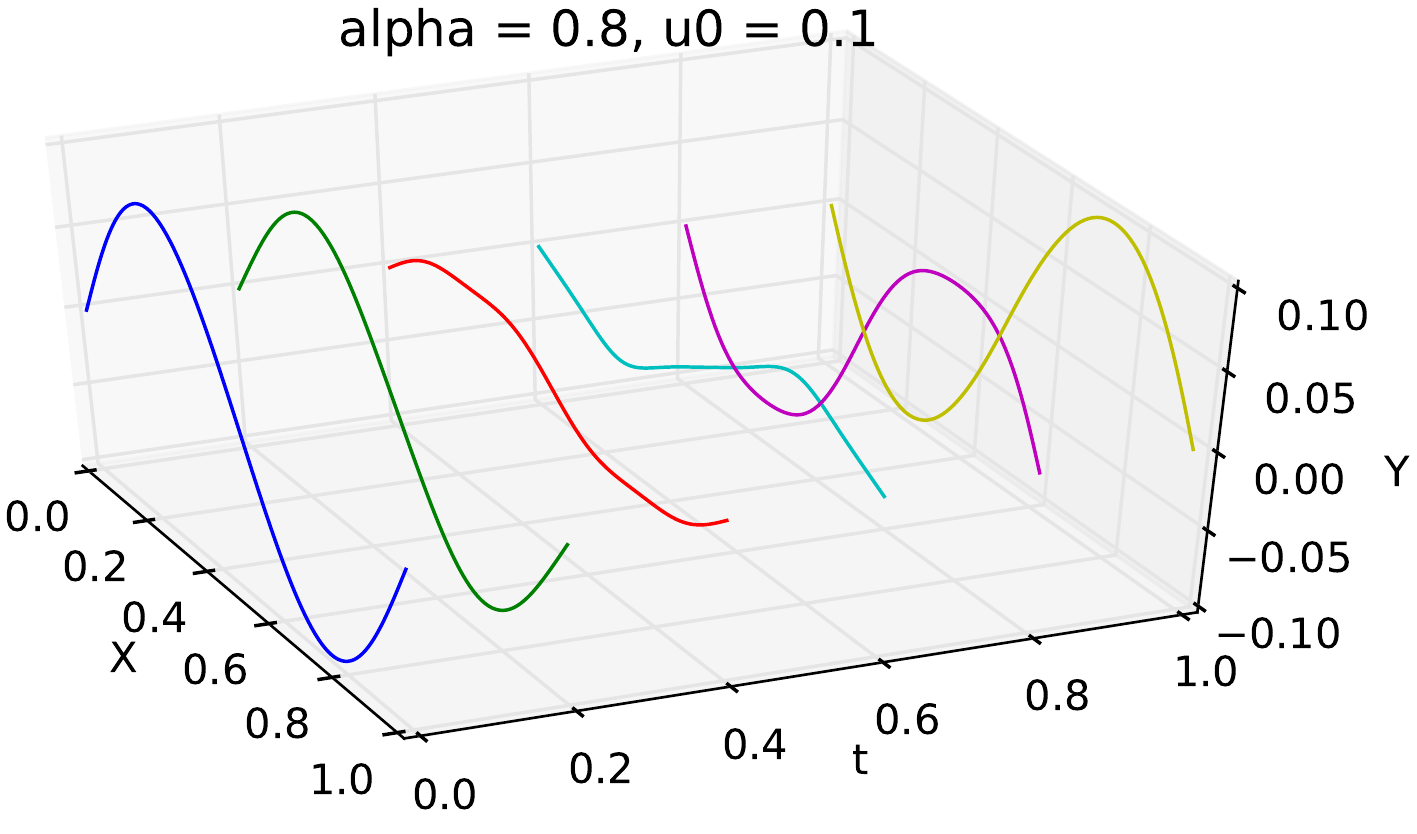}
}
\subfloat{
\includegraphics[width=0.33\textwidth]{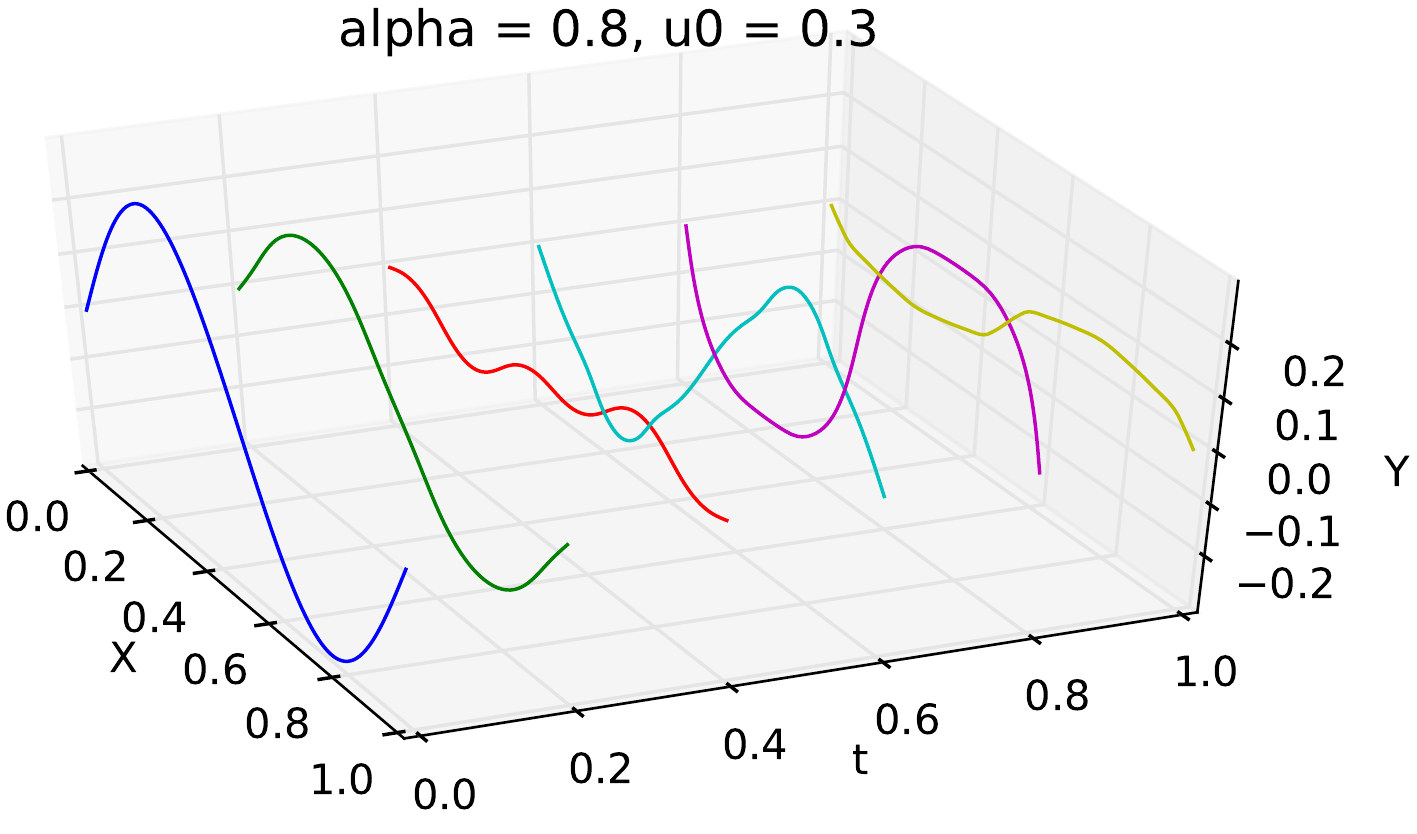}
}\\
\subfloat{
\includegraphics[width=0.33\textwidth]{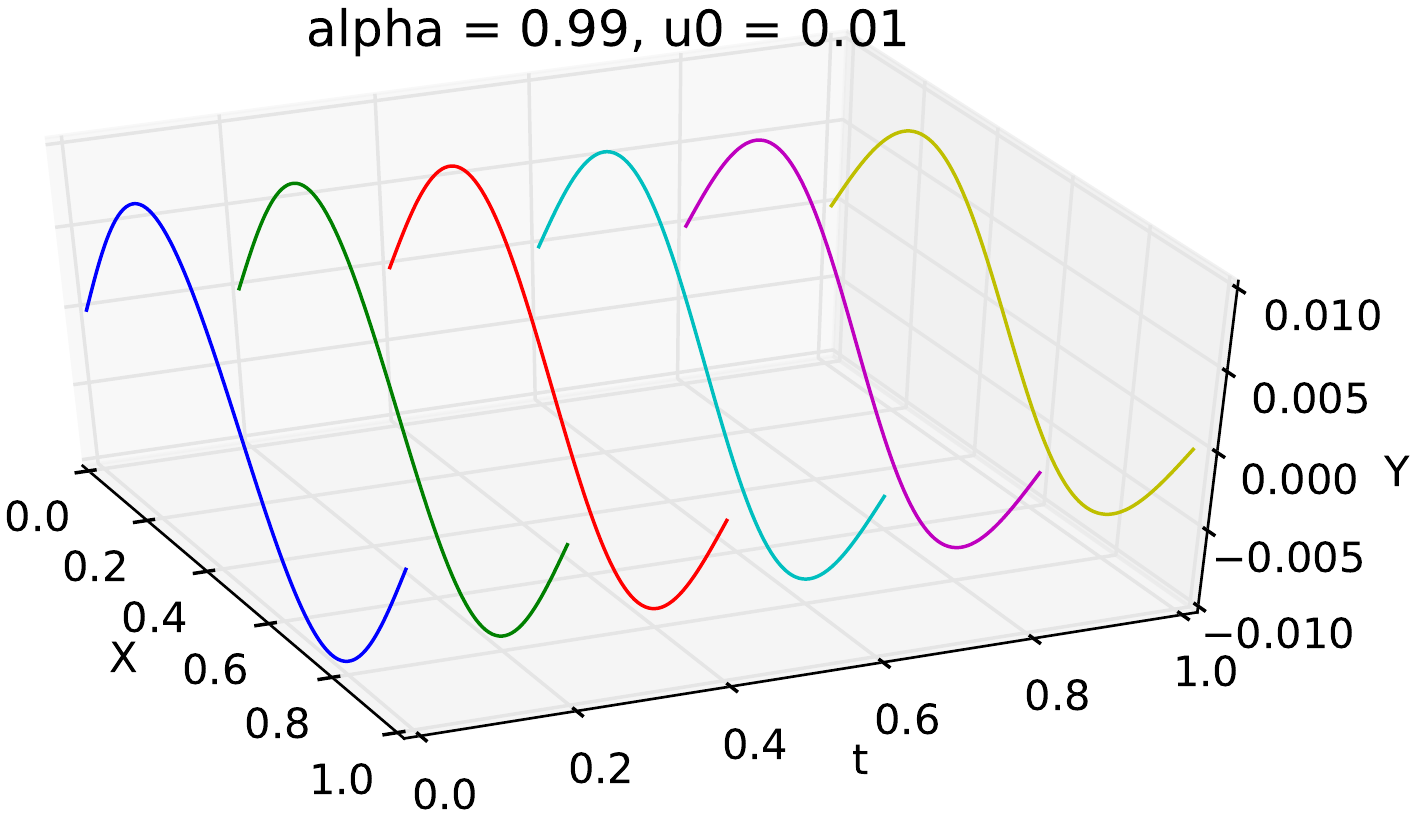}
}
\subfloat{
\includegraphics[width=0.33\textwidth]{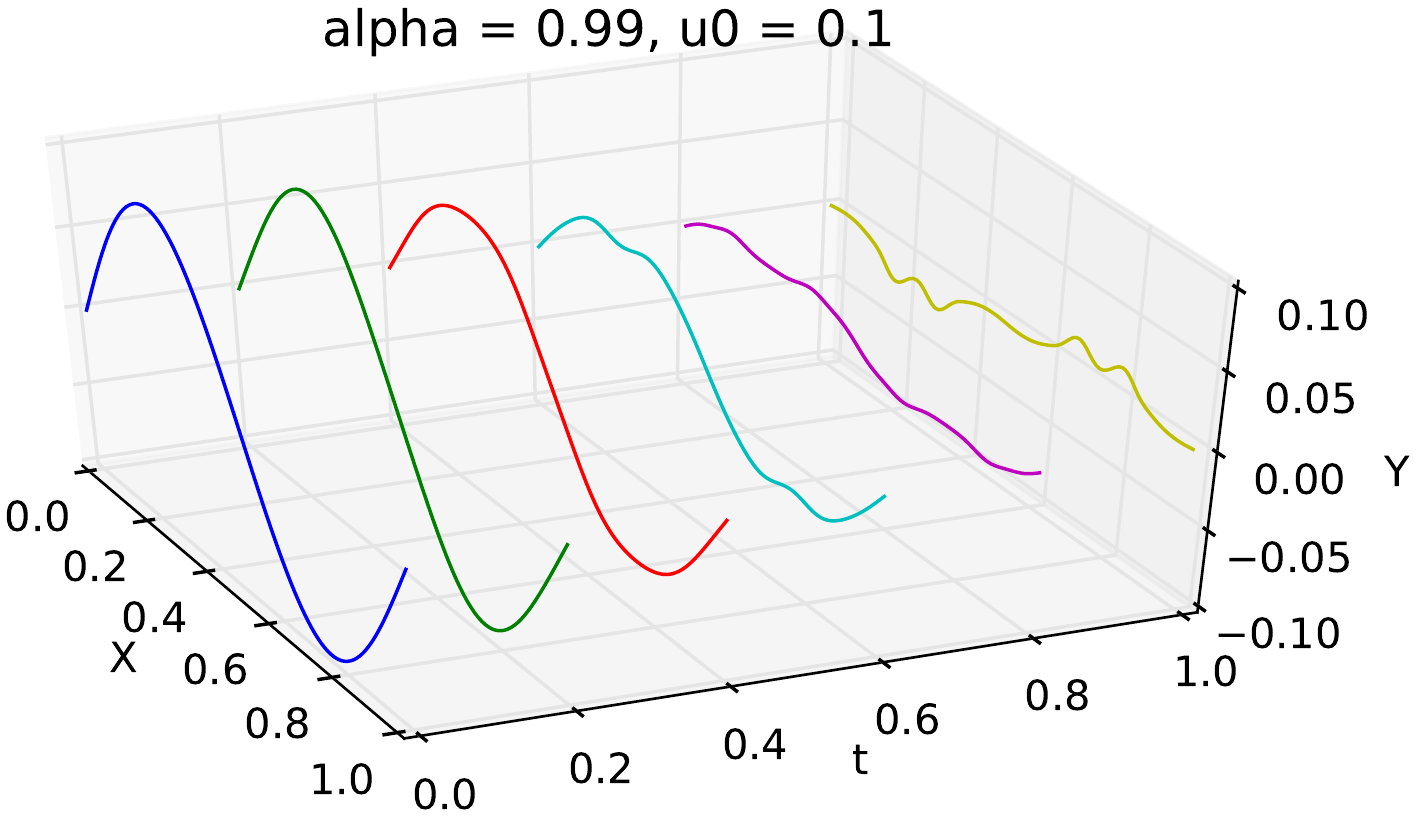}
}
\subfloat{
\includegraphics[width=0.33\textwidth]{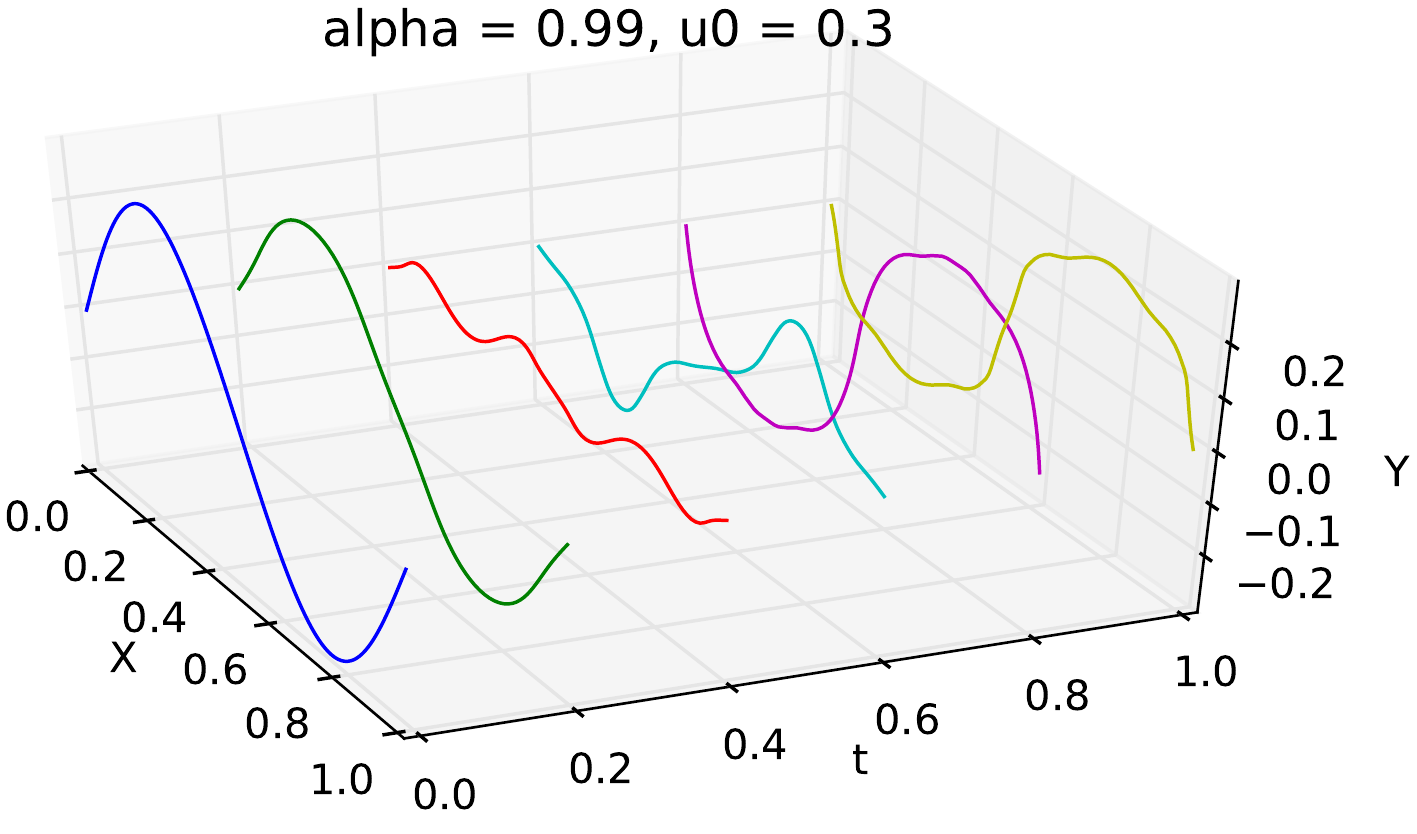}
}

\caption{Nonlinear wave equation: Deformations of the string over time
with nonlinearity parameter $\alpha=0$ (top), $\alpha=0.8$ (middle), and $\alpha=0.99$ (bottom); the amplitude of $u^0$ is $u_0=0.01$ (left), $u_0 = 0.1$ (middle), and $u_0=0.3$ (right)}
\label{defs}
\end{figure}

The time-variation of the discrete pseudo-energy $\tilde{H}^n$ defined by (\ref{discrete energy}) and
of the discrete energy $H^n$ defined by (\ref{eq:discrete energy bis})
are shown in Figure~\ref{energy variation modif et normale}
in the most challenging case where
$\alpha = 0.99$ and $u^0=0.3$. 
We first observe that the variations the discrete energy
$H^n$ are very moderate. In Figure \ref{energy variation modif et normale on fly}, we illustrate the adaptive time-stepping strategy discussed in Remark~\ref{rq:a posteriori}, where we take $\epsilon_{\textrm{fly}}=0.03\%$ to control the departure of $\tilde{H}^n$ from $H^n$ at each iteration. 
In this situation, the adaptive time steps take values in the range 
$[0.0024,0.005]$. Concerning the discrete
pseudo-energy $\tilde{H}^n$, we observe 
conservation up to machine precision when employing a five-point 
Gauss--Legendre quadrature of order 9.

% \begin{figure}[!htp]
% \begin{center}

% \label{comparison modif et normale}
% \end{center}
% \end{figure}

\begin{figure}[!htp]
\begin{center}
\subfloat[]{\label{energy variation modif et normale}
\resizebox{0.5\textwidth}{!}{\includegraphics[scale=1]{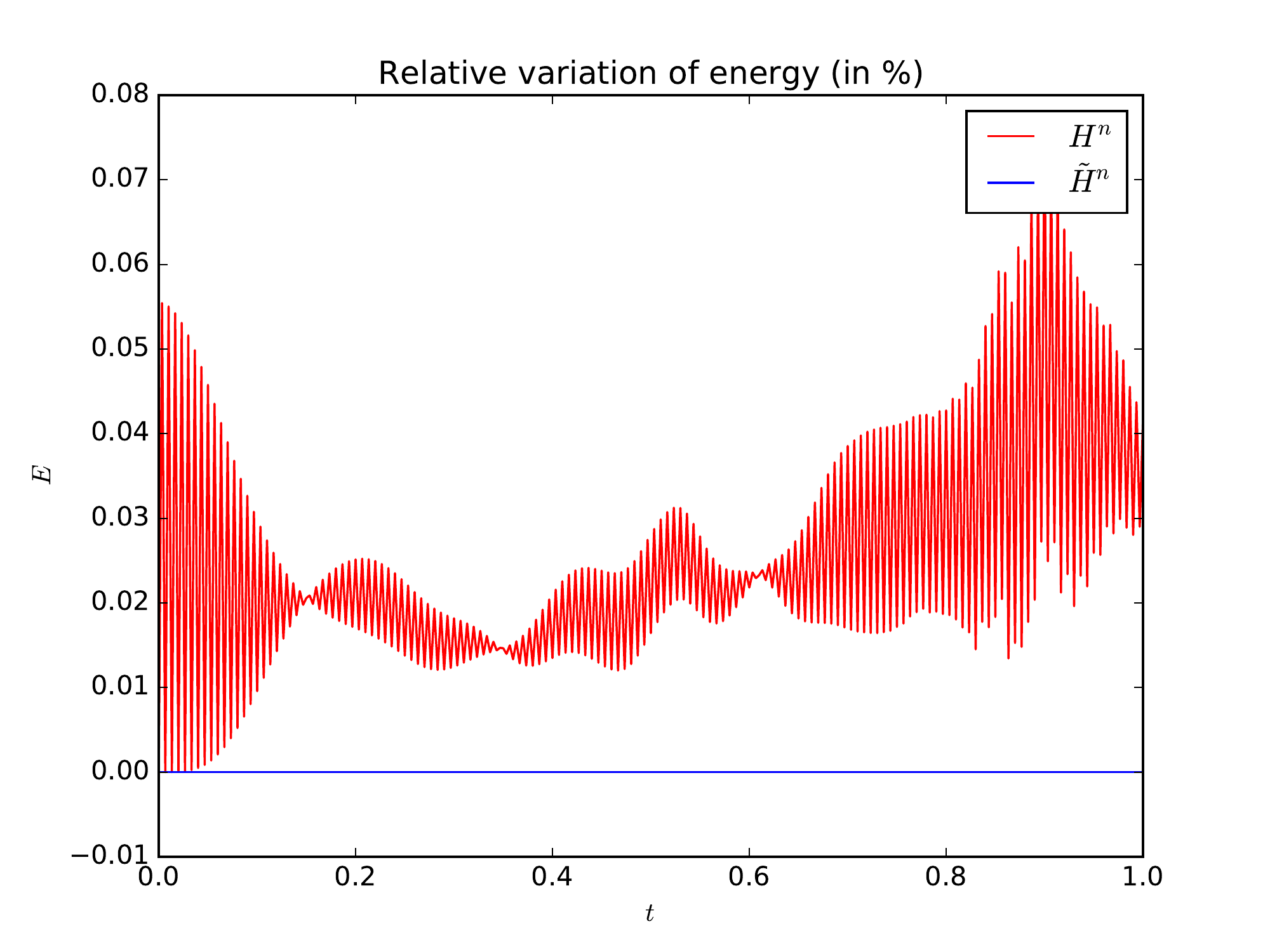}}
}
\subfloat[]{\label{energy variation modif et normale on fly}
\resizebox{0.5\textwidth}{!}{\includegraphics[scale=1]{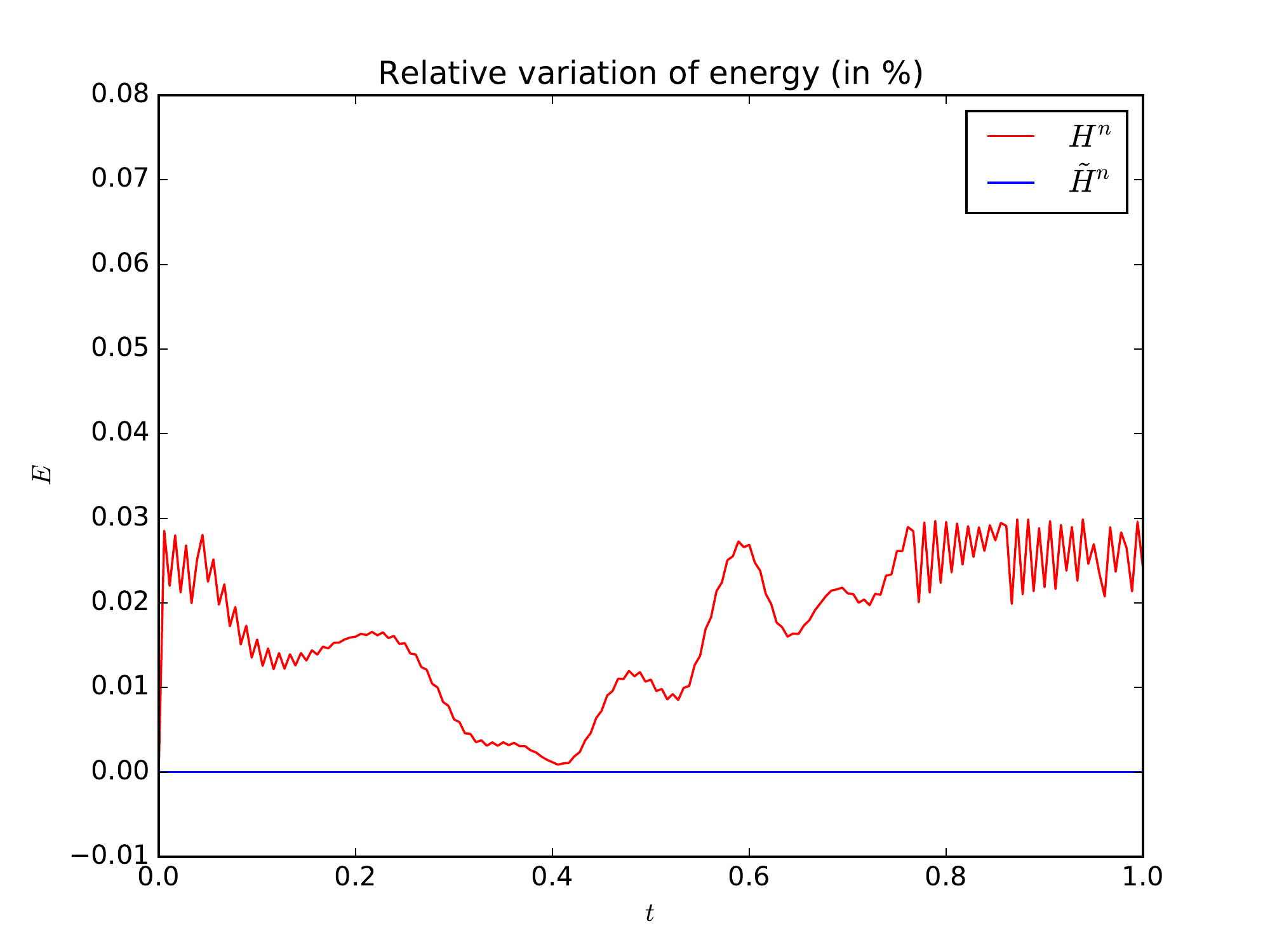}}
}
\caption{Nonlinear wave equation: time-variation of the discrete energy $H^n$ and pseudo-energy $\tilde{H}^n$ over a unit time interval for $\alpha = 0.99$ and an amplitude of $0.3$ for $u^0$: \protect\subref{energy variation modif et normale} Fixed time-step $\Delta t=0.0033$ ; \protect\subref{energy variation modif et normale on fly} Adaptive time-step as in Remark~\ref{rq:a posteriori}}
%\label{comparison modif et normale}
\end{center}
\end{figure}

%Sert pour valeurs absolues pour l'instant
%\begin{tabular}{ | l | c | c | c | c | c | c | c | c | c |}
%     \hline
%     $\alpha$ & \multicolumn{3}{|c|}{ $0$} &
%     \multicolumn{3}{|c|}{$0.8$} &
%     \multicolumn{3}{|c|}{$0.99$} \\ \hline
%     $u^0$ & $0.01$ &  $0.1$ & $0.3$ & $0.01$ &  $0.1$ & $0.3$ & $0.01$ &  $0.1$ & $0.3$\\ \hline
%     $E_0$ & 0.00099 & 0.099 & 0.89 & 0.00020 & 0.025 & 0.39 & 1.1e-05 & 0.0072 & 0.28 \\ \hline
%     MP & 6.5e-19 & 6.9e-17 & 5.6e-16 & 1.1e-11 & 1.4e-07 & 9.2e-06 & 1.6e-13 & 8.7e-08  & 3.1e-05
%\\ \hline
%     GL 3 & 6.5e-19 & 8.3e-17 & 5.6e-16 & 3.3e-17 & 3.1e-17 & 1.7e-14 & 3.6e-17 & 2.95e-17 & 1.8e-13
% \\ \hline
%     GL 5 & 6.5e-19 & 8.3e-17 & 7.8e-16 & 3.1e-17 & 3.1e-17 & 3.9e-16 & 3.4e-17 & 3.1e-17 & 2.2e-16
% \\
%   	 \hline
%   \end{tabular}

\begin{table}[!htp]
\begin{center}
   \begin{tabular}{ | l | c | c | c | c | c | c | c | c | c |}
     \hline

     $\alpha$ & \multicolumn{3}{|c|}{ $0$} &
     \multicolumn{3}{|c|}{$0.8$} &
     \multicolumn{3}{|c|}{$0.99$} \\ \hline
     $u^0$ & $0.01$ &  $0.1$ & $0.3$ & $0.01$ &  $0.1$ & $0.3$ & $0.01$ &  $0.1$ & $0.3$\\ \hline
     MP & me & me & me & 5.5e-08 & 5.6e-05 & 2.4e-05 & 1.5e-08 & 1.2e-05 & 1.1e-04
\\ \hline
     GL5 & me & me & me & 1.7e-13 & 1.2e-14 & 4.4e-14 & 3.3e-12 & 6.4e-13 & 1.8e-13
 \\ \hline
     GL9 & me & me & me & 1.6e-13 & 1.2e-14 & me & 3.1e-12 & me & me%2.2e-16
 \\
   	 \hline
   \end{tabular}
   \caption{Nonlinear wave equation: Maximal errors in the conservation of the pseudo-energy $\tilde{H}^n$ for the mid-point (MP), three-point Gauss--Legendre of order 5 (GL5) and five-point Gauss--Legendre of order 9 (GL9) quadratures; `me' means machine error}
   \label{tab:energy errors}
\end{center}
\end{table}
\begin{figure}[!htp]
\begin{center}
\input{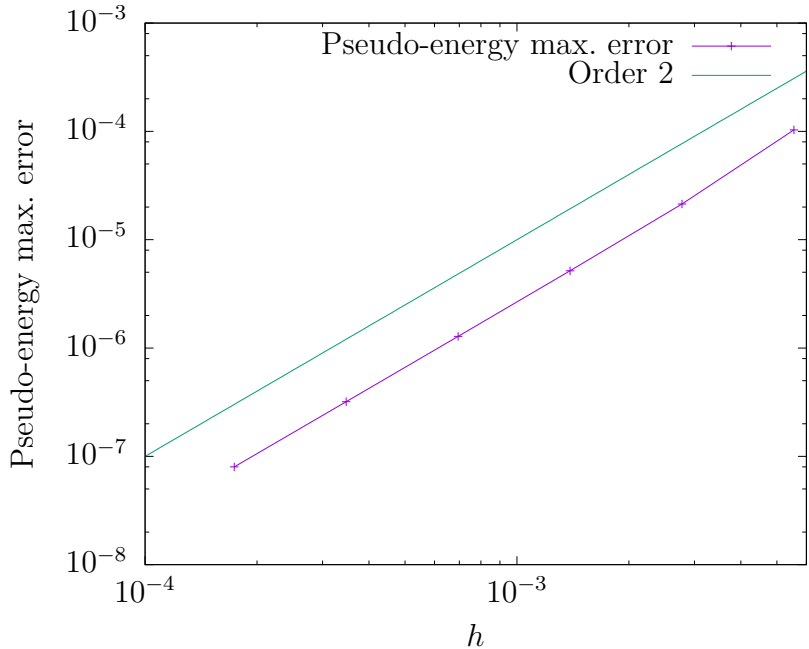}
\caption{Nonlinear wave equation: Maximal error on the conservation of the pseudo-energy $\tilde H^n$ as a function of the time-step in the case $\alpha = 0.99$, $u^0 = 0.3$, and a mid-point quadrature}
\label{fig:error_conv_energy}
\end{center}
\end{figure}

To illustrate
the impact of quadratures on pseudo-energy conservation, we perform 
two numerical experiments. First, Table \ref{tab:energy errors} reports the maximal variation of the discrete pseudo-energy $\tilde{H}^n$ depending on the quadrature used with a constant time-step $\Delta t = 0.0033$. The mid-point quadrature is precise enough when used on the linear equation ($\alpha = 0$). The three-point Gauss--Legendre quadrature of order 5 is found to give very satisfactory results for the two nonlinear cases ($\alpha=0.8$ and $\alpha=0.99$).
As announced in Remark \ref{rq:quadratures}, the maximal error on pseudo-energy conservation is observed to decrease when increasing the quadrature order. The total number of force evaluations is 300, 900, and 1500 when using the mid-point quadrature and the three- and five-point Gauss--Legendre quadratures of order 5 and 9, respectively.
In the second experiment, we illustrate the second-order accuracy of 
pseudo-energy conservation when using the mid-point quadrature. 
We consider again the most challenging case where $\alpha = 0.99$ and $u^0 = 0.3$.
Figure \ref{fig:error_conv_energy} shows the maximal variation of the discrete pseudo-energy $\tilde{H}^n$ with respect to the value of the (fixed) time-step used in the simulation, confirming the second-order accuracy.

To conclude this section, we present some comments on the relative costs of the present scheme with respect to an implicit scheme, e.g., the one devised in~\cite{chabassier2010energy}. In the most challenging case where $\alpha = 0.99$ and $u^0 = 0.3$, the energy conservation in \cite{chabassier2010energy} is machine error. For a comparable error, we can consider the present scheme with the five-point Gauss--Legendre quadrature of order 9. The convergence criterion of the Newton's method in the implicit scheme can be estimated to require at least a couple of iterations per time-step and the same number of Hessian computations and global matrix inversions per time-step, whereas the explicit method with the five-point Gauss--Legendre quadrature of order 9 requires only 5 force evaluations per time-step. Thus, although no general conclusions can be drawn, the present explicit method stands good chances to be quite competitive with respect to an implicit method.

\section{Asynchronous scheme}\label{sec:asynchronous}
Owing to the CFL condition (\ref{eq:abs_stab}), the time-step can be
required to be small in regions with stiff or nonsmooth dynamics. The
overall efficiency of the computation would be compromised by the
large number of integral calculations in the whole domain, while most
of these would be redundant in smooth regions. We therefore propose an
asynchronous version of the scheme which preserves the general
properties of the synchronous version. In this section, 
we first present the idea behind slow-fast decomposition of the particles
and we devise an asynchronous scheme for which we prove pseudo-energy conservation at the slow time nodes under exact force integration. Second-order
accuracy is expected and is illustrated numerically on two test cases
including an inhomogeneous wave equation.

\subsection{Slow-fast splitting}

In order to simplify the presentation of the asynchronous scheme, we
limit ourselves here to the integration of a slow-fast dynamics, i.e.,
we consider a system with essentially two distinct time scales. The
forces between the particles are supposed to be split into a "fast"
set with an associated time-step $h_F$ and a "slow" set with an
associated time-step $h_S > h_F$.  For example, the
splitting can result from the relative stiffness of the forces in
the system. Consequently, the particles are split into three sets: the slow particles are subjected only to slow forces, the fast particles are subjected only to fast forces, and the remaining particles, which are called mixed particles, are subjected to both slow and fast forces.
This definition means that the slow particles do not interact with the fast particles directly, so that the potential $V$ can be
decomposed as follows:
\[V(\bm{q}) = V_S(\bm{q}_S) + V_M(\bm{q}_M,\bm{q}_S) + V_F(\bm{q}_F,\bm{q}_M),\]
where $\bm{q}_F$, $\bm{q}_S$, and $\bm{q}_M$ denote
respectively the positions of the fast, slow, and mixed
particles, the potential $V_S$ describes the interactions between slow particles, $V_M$ the interactions between slow and mixed particles, and $V_F$ the interactions between mixed and fast particles (or between themselves). 
For instance, the purple particle in
Figure \ref{fig:rapides_lentes} and the Particle 3 in
Figure \ref{fig:asynchronous} are mixed particles. The mixed particle in Figure \ref{fig:rapides_lentes} is subjected to a "fast" force by the stiff spring on its right and a "slow" force by the soft spring on its left. The fast particle (in red) is only subjected to a "fast" force by the stiff spring. The slow particles (in blue) are only subjected to "slow" forces by the two soft springs. Finding a slow-fast decomposition is not possible for every Hamiltonian system. For example, in the case where all the particles interact with each other, no slow-fast splitting is available. The most favorable configuration is the one where the slow and the fast particles interact essentially among themselves and have very few interactions with mixed particles. This configuration is encountered in inhomogeneous problems where an interface separates two zones where the properties are different; the slow and the fast particles are then located in the two zones, whereas the mixed particles are located at the interface.

In what follows, 
we abuse the notation by denoting 
$F$, $M$ and $S$ the sets collecting the indices in 
$\{1,\ldots,N\}$ of the fast, mixed
and slow particles, respectively. For simplicity, we assume
that the mass matrix $\bm{M}$ is diagonal and denote
$\bm{M}_F$, $\bm{M}_M$ and $\bm{M}_S$ the restriction of
$\bm{M}$ to the $F$, $M$ and $S$ particles respectively.
Still for simplicity, we assume that
both time-steps $h_S$ and $h_F$ are kept constant. 

\begin{figure}[!htp]
\begin{center}
\begin{tikzpicture}[x=2cm,y=1cm]
\def\coilfineup#1{
        {#1 +  (1.5-0.4)/7*\t - 0.2*(cos(\t * pi r)-1)},
        {0.4 * sin(\t * pi r)}
        }
\def\coilfinedown#1{
        {#1 +  (1.5-0.4)/7*\t - 0.2*(cos(\t * pi r)-1)},
        {-0.4 * sin(\t * pi r)}
        }
\def\coilthickup#1{
        {#1 +  (1.5-0.2)/15*\t - 0.1*(cos(\t * pi r)-1)},
        {0.4 * sin(\t * pi r)}
        }
\def\coilthickdown#1{
        {#1 +  (1.5-0.2)/15*\t - 0.1*(cos(\t * pi r)-1)},
        {-0.4 * sin(\t * pi r)}
        }
\draw[domain={0:7},smooth,variable=\t,samples=100] plot (\coilfineup{0});
\draw[domain={0:7},smooth,variable=\t,samples=100] plot (\coilfinedown{1.5});
\draw[domain={0:15},smooth,variable=\t,samples=100, thick] plot (\coilthickup{3});
%\draw[domain={0:15},smooth,variable=\t,samples=100, thick] plot (\coilthickdown{4.5});
\draw[fill, blue] (0,0) circle [radius=3pt];
\draw[fill, blue] (1.5,0) circle [radius=3pt];
\draw[fill, blue!50!red] (3,0) circle [radius=3pt];
\draw[fill, red] (4.5,0) circle [radius=3pt];
%\draw[fill, red] (6,0) circle [radius=3pt];
\draw (0,0.4) node[anchor=south]{slow};
\draw (1.5,0.4) node[anchor=south]{slow};
\draw (3,0.4) node[anchor=south]{mixed};
\draw (4.5,0.4) node[anchor=south]{fast};
%\draw (6,0.4) node[anchor=south]{fast};
\draw (0.75,-0.4) node[anchor=north]{\begin{minipage}{2cm}\centering
        soft \end{minipage}};
\draw (2.25,-0.4) node[anchor=north]{\begin{minipage}{2cm}\centering
        soft \end{minipage}};
\draw (3.75,-0.4) node[anchor=north]{\begin{minipage}{2cm}\centering
        stiff \end{minipage}};
%\draw (5.25,-0.4) node[anchor=north]{\begin{minipage}{2cm}\centering
%        stiff \end{minipage}};
\end{tikzpicture}
\end{center}
\caption{Example of system of particles with a slow-fast splitting}
\label{fig:rapides_lentes}
\end{figure}
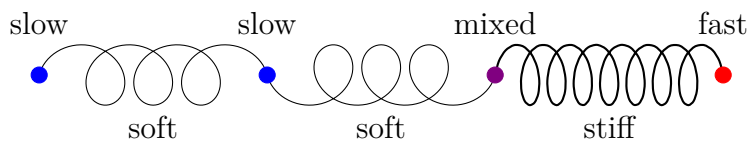

\subsection{Presentation of the asynchronous scheme}

Without much loss of generality, we can suppose that the slow and fast
time-steps are commensurate so that $h_S = K h_F$ with
$K\in\mathbb{N}^*$. We then define the coarse
time nodes $t^n=nh_S$ and the fine time nodes $t^{n,m}=t^n+mh_F$ for all
$m\in\{0,\ldots,K\}$. 
The asynchronous scheme consists in integrating $K$ times the
dynamics of the $F$ and $M$ particles with the "fast" forces computed
at each time-step of length $h_F$ and in updating the $S$ particles 
with the "slow''
forces computed once at the end of each time-step of length $h_S$. The
general procedure is depicted in Figure \ref{fig:asynchronous} for
four particles in the same configuration as in
Figure \ref{fig:rapides_lentes}. The efficiency of the asynchronous
scheme hinges on the fact that each particle has a free-flight
movement during each time-step, with the neighbouring particle forces
acting only at the end of the time-step. 

\begin{figure}[!htp]
\centering
\begin{tikzpicture}[x=5cm,y=2cm]
\draw (-0.6,0) node[anchor=south,rounded corners=3pt,draw=black]{Soft};
\draw (0.5,0) node[anchor=south,rounded corners=3pt,draw=black]{Soft};
\draw (1.5,0) node[anchor=south,rounded
corners=3pt,draw=black]{Stiff};
%% \draw (2.5,0) node[anchor=south,rounded
%% corners=3pt,draw=black]{Stiff};
%Particule 1
\begin{scope}[blue]
\draw (-1,0.5) node[anchor=south]{Particle 1};
\draw (-1,0.25) node[rounded corners=3pt,draw=black]{Slow};
\draw (-1,0) node{$\bullet$};
\draw[thick,-latex] (-1,0) -- (-1,-4);
%Particule 2
\draw (-0.2,0.5) node[anchor=south]{Particle 2};
\draw (-0.2,0.25) node[rounded corners=3pt,draw=black]{Slow};
\draw (-0.2,0) node{$\bullet$};
\draw[thick,-latex] (-0.2,0) -- (-0.2,-4);
\end{scope}
%Particule 3
\begin{scope}[blue!50!red]
\draw (1,0.5) node[anchor=south]{Particle 3};
\draw (1,0.25) node[rounded corners=3pt,draw=black]{Mixed};
\draw (1,0) node{$\bullet$};
\draw[thick,-latex] (1,0) -- (1,-1);
\draw[thick,-latex] (1,-1) -- (1,-2);
\draw[thick,-latex] (1,-2) -- (1,-3);
\draw[thick,-latex] (1,-3) -- (1,-4);
\end{scope}
%Particule 4
\begin{scope}[red]
\draw (2,0.5) node[anchor=south]{Particle 4};
\draw (2,0.25) node[rounded corners=3pt,draw=black]{Fast};
\draw (2,0) node{$\bullet$};
\draw[thick,-latex] (2,0) -- (2,-1);
\draw[thick,-latex] (2,-1) -- (2,-2);
\draw[thick,-latex] (2,-2) -- (2,-3);
\draw[thick,-latex] (2,-3) -- (2,-4);
\end{scope}
%Pas de temps
\draw (-1,-2) node[anchor=east]{$h_S$};
\draw (2,-0.5) node[anchor=west]{$h_F$};
\draw (-1,0) node[anchor=east]{$t^n$};
\draw (-1,-4) node[anchor=east]{$t^{n+1}$};
\draw (2,0) node[anchor=west]{$t^{n,0}$};
\draw (2,-1) node[anchor=west]{$t^{n,1}$};
\draw (2,-2) node[anchor=west]{$t^{n,2}$};
\draw (2,-3) node[anchor=west]{$t^{n,3}$};
\draw (2,-4) node[anchor=west]{$t^{n,4}$};
\draw[dotted] (-1,0) -- (2,0);
\draw[dotted] (0.1,-1) -- (2,-1);
\draw[dotted] (0.1,-2) -- (2,-2);
\draw[dotted] (0.1,-3) -- (2,-3);
\draw[dotted] (-1,-4) -- (2,-4);
%Integration numerique
% Lien 1-2
\draw[latex-latex] (-1,-2) -- (-0.2,-2);
\draw (-0.6,-2) node[rounded corners=3pt,fill=white]{$\displaystyle\int_{0}^{h_S}{\nabla V_S(\bm{\hat{q}}_{S}^n)}$};
% Lien 2-3
\draw[latex-latex] (-0.2,-2) -- (0,-2) -- (0,-0.5) -- (1,-0.5);
\draw (0.5,-0.5) node[rounded corners=3pt,fill=white]{$\displaystyle\int_{0}^{h_F}{\nabla V_M(\bm{\hat{q}}^n_{S},\bm{\hat{q}}^{n,0}_M)}$};
\draw[latex-latex] (-0.2,-2) -- (0,-2) -- (0,-1.5) -- (1,-1.5);
\draw (0.5,-1.5) node[rounded corners=3pt,fill=white]{$\displaystyle\int_{h_F}^{2h_F}{\nabla V_M(\bm{\hat{q}}^n_{S},\bm{\hat{q}}^{n,1}_M)}$};
\draw[latex-latex] (-0.2,-2) -- (0,-2) -- (0,-2.5) -- (1,-2.5);
\draw (0.5,-2.5) node[rounded corners=3pt,fill=white]{$\displaystyle\int_{2h_F}^{3h_F}{\nabla V_M(\bm{\hat{q}}^n_{S},\bm{\hat{q}}^{n,2}_M)}$};
\draw[latex-latex] (-0.2,-2) -- (0,-2) -- (0,-3.5) -- (1,-3.5);
\draw (0.5,-3.5) node[rounded corners=3pt,fill=white]{$\displaystyle\int_{3h_F}^{4h_F}{\nabla V_M(\bm{\hat{q}}^n_{S},\bm{\hat{q}}^{n,3}_M)}$};
% Lien 3-4
\draw[latex-latex] (1,-0.5) -- (2,-0.5);
\draw (1.5,-0.5) node[rounded corners=3pt,fill=white]{$\displaystyle\int_{0}^{h_F}{\nabla V_F(\bm{\hat{q}}^{n,0}_{F},\bm{\hat{q}}^{n,0}_M)}$};
\draw[latex-latex] (1,-1.5) -- (2,-1.5);
\draw (1.5,-1.5) node[rounded corners=3pt,fill=white]{$\displaystyle\int_{h_F}^{2h_F}{\nabla V_F(\bm{\hat{q}}^{n,1}_{F},\bm{\hat{q}}^{n,1}_M)}$};
\draw[latex-latex] (1,-2.5) -- (2,-2.5);
\draw (1.5,-2.5) node[rounded corners=3pt,fill=white]{$\displaystyle\int_{2h_F}^{3h_F}{\nabla V_F(\bm{\hat{q}}^{n,2}_{F},\bm{\hat{q}}^{n,2}_M)}$};
\draw[latex-latex] (1,-3.5) -- (2,-3.5);
\draw (1.5,-3.5) node[rounded corners=3pt,fill=white]{$\displaystyle\int_{3h_F}^{4h_F}{\nabla V_F(\bm{\hat{q}}^{n,3}_{F},\bm{\hat{q}}^{n,3}_M)}$};
\end{tikzpicture}
\caption{Asynchronous integration of four particles with a slow-fast synamics, $h_S = 4h_F$}
\label{fig:asynchronous}
\end{figure}
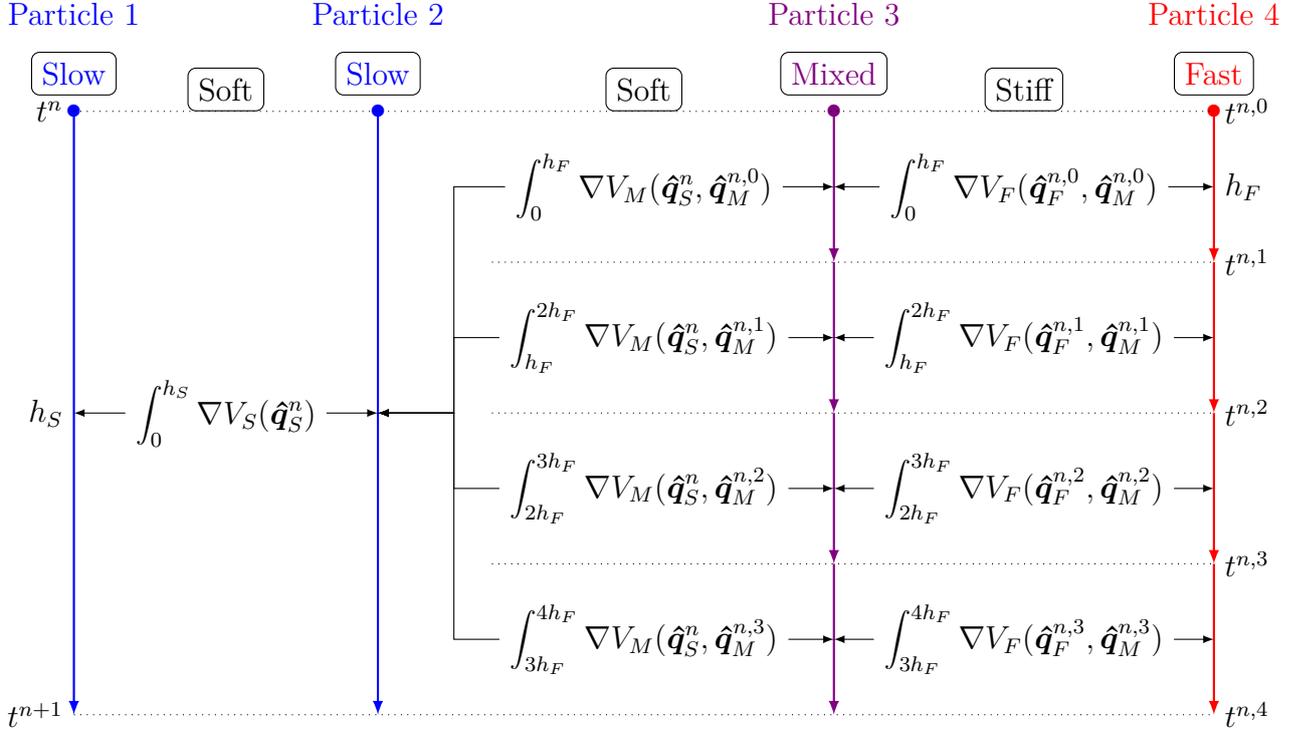

Let us now describe in more detail the asynchronous scheme over the coarse
time interval $I_n=[t^n,t^{n+1}]$. At the beginning, we 
have at our disposal the triple
$(p_i^{n-1/2},q_i^n,p_i^{n+1/2})$ for the slow particles ($i\in S$) and the 
triple $(p_i^{n,-1/2}=p_i^{n-1,K-1/2},q_i^{n,0},p_i^{n,1/2})$ for the fast and the mixed particles 
($i\in F\cup M$). The asynchronous scheme then proceeds as follows (we use here the two-step formulation which reduces to \eqref{schema multistep} in the synchronous case):
\begin{itemize}
\item For the fast particles ($i\in F$), one computes for all $m\in \{0,\ldots,K-1\}$,
\begin{subequations}
\begin{align}
q_i^{n,m+1} &= q_i^{n,m} + h_F\frac{1}{m_i}p_i^{n,m+1/2}, \label{eq:fast_position}\\
p_i^{n,m+3/2} &= p_i^{n,m-1/2}
-2\int_{t^{n,m}}^{t^{n,m+1}}{ \frac{\partial V_F}{\partial q_i}(\bm{\hat{q}}_{F}^{n,m}(t),\bm{\hat{q}}_{M}^{n,m}(t)) dt},\label{eq:fast_jump}
\end{align}\label{eq:fast}
\end{subequations}
with the free-flight trajectories for the fast and the mixed particles defined as
\begin{equation}
\hat{q}_{j}^{n,m}(t) = q_j^{n,m}
+\frac{1}{m_j}p_j^{n,m+1/2}(t-t^{n,m}),\quad \forall t\in[t^{n,m},t^{n,m+1}],
\; \forall j\in F\cup M.
\end{equation}
\item For the mixed particles ($i\in M$), 
one computes for all $m\in \{0,\ldots,K-1\}$, the position $q_i^{n,m+1}$ as in \eqref{eq:fast_position}, whereas equation (\ref{eq:fast_jump}) is
replaced by
\begin{equation}
p_i^{n,m+3/2} = p_i^{n,m-1/2}
-2\int_{t^{n,m}}^{t^{n,m+1}}{ \left(\frac{\partial V_F}{\partial q_i}(\bm{\hat{q}}_{F}^{n,m}(t),\bm{\hat{q}}_{M}^{n,m}(t)) + \frac{\partial V_M}{\partial q_i}(\bm{\hat{q}}_{M}^{n,m}(t),\bm{\hat{q}}_{S}^n(t))\right) dt},\label{eq:mixed_jump}
\end{equation}
where the free-flight trajectories of the slow particles are computed over the coarse time interval as follows:
\begin{equation} \label{eq:free_flight_asyn_S}
\hat{q}_{j}^{n}(t) = q_j^{n}
+\frac{1}{m_{j}}p_{j}^{n+1/2}(t-t^{n}),\quad \forall t\in[t^{n},t^{n+1}],
\; \forall j\in S.
\end{equation}
\item For the slow particles ($i\in S$), one computes
\begin{subequations}
\begin{align}
q_i^{n+1} & = q_i^n + h_S\frac{1}{m_i}p_i^{n+1/2},\\
p_i^{n+3/2} &= p_i^{n-1/2}
-2\sum_{m=0}^{K-1}\int_{t^{n,m}}^{t^{n,m+1}}\frac{\partial V_M}{\partial
q_i}(\bm{\hat{q}}_{M}^{n,m}(t),\bm{\hat{q}}_{S}^n(t))dt -2\int_{t^n}^{t^{n+1}} \frac{\partial V_S}{\partial q_i}(\bm{\hat{q}}_{S}^n(t)) dt,
\end{align}
\label{eq:slow}
\end{subequations}
with the free-flight trajectories defined above.
\end{itemize}

Note that the slow forces between slow and
mixed particles need to be evaluated at every fine
time-step. In the
worst case scenario, every slow force links a slow particle with a
mixed particle, which results in the asynchronous scheme reverting to
the synchronous scheme. Such a case typically occurs when the particles
all interact or when the system alternates fast and slow forces. On
the other hand, the efficiency of the asynchronous scheme compared to
the synchronous scheme is maximal in the case where the mixed particles
constitute a small fraction of the particles and their interaction
is limited to a small fraction of the slow particles. A typical case is a
nearest-neighbour interaction with slow and fast particles located in
distinct regions, the mixed particles being confined in a lower
dimensional delimiting interface. In the limit of a large number of
particles, the computational cost per large time-step $h_S$
reduces to $K$ integrals of the fast forces and one integral of the
slow forces.

\begin{proposition} [Synchronization of particles]
\label{resynchro energy}
Assume that the numerical integration is exact. Then the numerical scheme
(\ref{eq:fast})--(\ref{eq:slow}) exactly conserves the following pseudo-energy at the coarse time nodes $t^n$:
\begin{align}
\tilde{H}^{n} = {}&V_S(\bm{q}_S^n) + V_M(\bm{q}_M^{n,0},\bm{q}_S^{n}) + V_F(\bm{q}_F^{n,0},\bm{q}_M^{n,0})
 \nonumber\\
&+ \sum_{i\in S} \frac{1}{2m_i} \left(p_{i}^{n-1/2}\right)^{\mathrm{T}} p_i^{n+1/2}+ \sum_{i\in F\cup M} \frac{1}{2m_i} \left(p_i^{n,-1/2}\right)^{\mathrm{T}} p_i^{n,1/2}. \label{discrete energy asynchronous}
\end{align}
\end{proposition}

\begin{proof}
Let us set
\begin{align*}
\tilde{H}_S^n =&  V_S(\bm{q}_S^n) + \sum_{i\in S} \frac{1}{2m_i} \left(p_{i}^{n-1/2}\right)^{\mathrm{T}} p_i^{n+1/2}, \\
\tilde{H}_{FM}^{n,m} =& V_F(\bm{q}_F^{n,m},\bm{q}_M^{n,m}) +
V_M(\bm{q}_M^{n,m},\bm{\hat{q}}_S^n(t^{n,m})) + \sum_{i\in F\cup M} \frac{1}{2m_i} \left(p_i^{n,m-1/2}\right)^{\mathrm{T}} p_i^{n,m+1/2},
\end{align*}
for all $m\in \{0,\ldots,K\}$,
so that $\tilde{H}^n = \tilde{H}_S^n+\tilde{H}_{FM}^{n,0}$.
Following the same calculations as in the proof of Theorem \ref{energy
conservation} for equation (\ref{eq:slow}), we infer that
\begin{equation*} 
\tilde{H}_S^{n+1} = \tilde{H}_S^n
- \sum_{m=0}^{K-1}\int_{t^{n,m}}^{t^{n,m+1}}{\frac{\partial
V_M}{\partial \bm{q}_S}(\bm{\hat{q}}_M^{n,m}(t),\bm{\hat{q}}_{S}^n(t))\cdot\left(\bm{M}_{S}^{-1}\bm{p}_{S}^{n+1/2}\right)dt}.
\end{equation*}
Similarly, for all $m\in\{0,\dots,K-1\}$, using (\ref{eq:fast}) and (\ref{eq:mixed_jump}), we have
\begin{equation*}
\tilde{H}_{FM}^{n,m+1} = \tilde{H}_{FM}^{n,m}
+ \int_{t^{n,m}}^{t^{n,m+1}}{\frac{\partial V_M}{\partial \bm{q}_S}(\bm{\hat{q}}_M^{n,m}(t),\bm{\hat{q}}_{S}^n(t))\cdot\left(\bm{M}_{S}^{-1}\bm{p}_{S}^{n+1/2}\right)dt},
\end{equation*}
and summing over $m$, we obtain
\begin{equation*}
\tilde{H}_{FM}^{n+1,0} = \tilde{H}_{FM}^{n,K} = \tilde{H}_F^{n,0}
+ \sum_{m=0}^{K-1}{\int_{t^{n,m}}^{t^{n,m+1}}{\frac{\partial V_M}{\partial\bm{q}_S}(\bm{\hat{q}}_M^{n,m}(t),\bm{\hat{q}}_{S}^n(t))\cdot\left(\bm{M}_{S}^{-1}\bm{p}_{S}^{n+1/2}\right)dt}},
\end{equation*}
which gives the result.
\end{proof}

\begin{remarque} [Asynchronous pseudo-energy conservation]
The pseudo-energy $\tilde{H}^n$ of Theorem \ref{energy conservation} is not
conserved after every integration over a fast time-step $h_F$ in the
asynchronous setting. This results from the fact that during a "slow" time-step $h_S$, the effect of forces has been taken into account for the "fast" particles but not for the "slow" particles.
\end{remarque}

\subsection{Numerical results}

In this section, we present numerical results on the asynchronous scheme. We first consider a variant of the Fermi--Pasta--Ulam system with a slow-fast dynamics and then an inhomogeneous wave propagation problem.

\subsubsection{Fermi--Pasta--Ulam system with slow-fast dynamics}\label{sec:slow_fast}

We propose a slight variation of the Fermi--Pasta--Ulam test case in
order to assess the efficiency of the asynchronous
scheme. Contrary to the usual setting where stiff and soft
springs alternate, we suppose here that the system is composed of one
stiff region and one soft region, delimited by an interface in the
middle of the domain. Figure \ref{schema_slow_fast} illustrates the setting. There are $(m-1)$ fast particles, $1$ mixed particle, and $m$ slow particles.
\begin{figure}[!htp]
\begin{center}
\begin{tikzpicture}
\draw[fill=black!10!white, draw=none] (-0.7,-1.5) rectangle (0.,1.5);
\draw[pattern=north east lines, pattern color=black, draw=none] (-0.7,-1.5) rectangle (0.,1.5);
\draw[thick] (0,-1.5) -- (0,1.5);
\def\coilfineup#1{
        {#1 +  (1.5-0.4)/7*\t - 0.2*(cos(\t * pi r)-1)},
        {0.4 * sin(\t * pi r)}
        }
\def\coilfinedown#1{
        {#1 +  (1.5-0.4)/7*\t - 0.2*(cos(\t * pi r)-1)},
        {-0.4 * sin(\t * pi r)}
        }
\def\coilthickup#1{
        {#1 +  (1.5-0.2)/15*\t - 0.1*(cos(\t * pi r)-1)},
        {0.4 * sin(\t * pi r)}
        }
\def\coilthickdown#1{
        {#1 +  (1.5-0.2)/15*\t - 0.1*(cos(\t * pi r)-1)},
        {-0.4 * sin(\t * pi r)}
        }    
\draw[domain={0:15},smooth,variable=\t,samples=100] plot (\coilthickup{0});
\draw[domain={0:15},smooth,variable=\t,samples=100] plot (\coilthickdown{1.5});
\draw[domain={0:15},smooth,variable=\t,samples=100] plot (\coilthickup{3});
\draw[domain={0:7},smooth,variable=\t,samples=100, thick] plot (\coilfinedown{4.5});
\draw[domain={0:7},smooth,variable=\t,samples=100, thick] plot (\coilfineup{6});
\draw[domain={0:7},smooth,variable=\t,samples=100, thick] plot (\coilfinedown{7.5});
\draw[domain={0:7},smooth,variable=\t,samples=100, thick] plot (\coilfineup{9});
\draw[fill=black!10!white, draw=none] (10.5,-1.5) rectangle (11.2,1.5);
\draw[pattern=north east lines, pattern color=black, draw=none] (10.5,-1.5) rectangle (11.2,1.5);
\draw[thick] (10.5,-1.5) -- (10.5,1.5);
\draw[fill] (0,0) circle [radius=3pt];
\draw[fill] (1.5,0) circle [radius=3pt];
\draw[fill] (3,0) circle [radius=3pt];
\draw[fill] (4.5,0) circle [radius=3pt];
\draw[fill] (6,0) circle [radius=3pt];
\draw[fill] (7.5,0) circle [radius=3pt];
\draw[fill] (9,0) circle [radius=3pt];
\draw[fill] (10.5,0) circle [radius=3pt];
\draw (1.5,0.4) node[anchor=south]{$q_1$};
\draw (3,0.4) node[anchor=south]{$q_2$};
\draw (4.5,0.4) node[anchor=south]{$q_m$};
\draw (5.25,0.4) node[anchor=south]{$\cdots$};
\draw (7.5,0.4) node[anchor=south]{$q_{2m-1}$};
\draw (9,0.4) node[anchor=south]{$q_{2m}$};
\draw (2.25,-0.4) node[anchor=north]{\begin{minipage}{2cm}\centering
        stiff harmonic\end{minipage}};
\draw (6.75,-0.4) node[anchor=north]{\begin{minipage}{2cm}\centering
        soft nonlinear\end{minipage}};
\end{tikzpicture}
\end{center}
\caption{Setting for the Fermi--Pasta--Ulam system with slow-fast dynamics}
\label{schema_slow_fast}
\end{figure}
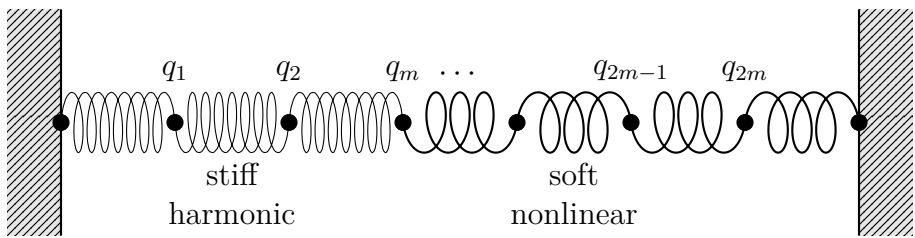

We consider a problem in dimension $d=1$. The Hamiltonian is given by
\[H(\bm{p},\bm{q}) = \frac{1}{2} \sum_{i=1}^{2m}{p^2_{i}}
+ \frac{\omega^2}{4} \sum_{i=1}^m (q_{i} - q_{i-1})^2
+ \sum_{i=m}^{2m}{(q_{i+1} - q_{i})^4}. \]
In the present experiment,
we take $m=3$ and $\omega^2 = 10$. The fast forces being generated by stiff linear springs, the fast time-step $h_F$ should respect the CFL condition from Equation (\ref{eq:abs_stab}), which here leads to $h_F <  3\cdot 10^{-3}$. 
% $h_F\leq 2\cdot 10^{-3}$
The small time-step $h_S$ being controlled by soft nonlinear springs, the CFL condition (\ref{eq:abs_stab}) is not applicable. A constant stable time-step has been found empirically to be $h_S\leq 10^{-1}$ using the five-point Gauss--Lobatto quadrature of order 7.
The dynamics
of the particles is presented in Figure \ref{fig:Position_slow_fast}
for $h_S = 0.01$ and $h_F=2\cdot 10^{-4}$, so that 50 iterations
of the fine time-step are carried out for each iteration of the coarse
time-step. Observe that, as expected, the fast particles ($1\leq i\leq
m$) exhibit oscillations with a typical frequency $\omega$, whereas the
slow particles ($m+1\leq i\leq 2m$) have tame nonlinear oscillations
with a frequency smaller than 1. Figure \ref{fig:Energy_slow_fast}
shows that the conservation of the discrete pseudo-energy  $\tilde H^n$ defined by (\ref{discrete energy}) is as perfect for the asynchronous
scheme as for the synchronous scheme with the five-point Gauss--Lobatto
quadrature of order 7.

\begin{figure}[!htp]
\begin{center}
\input{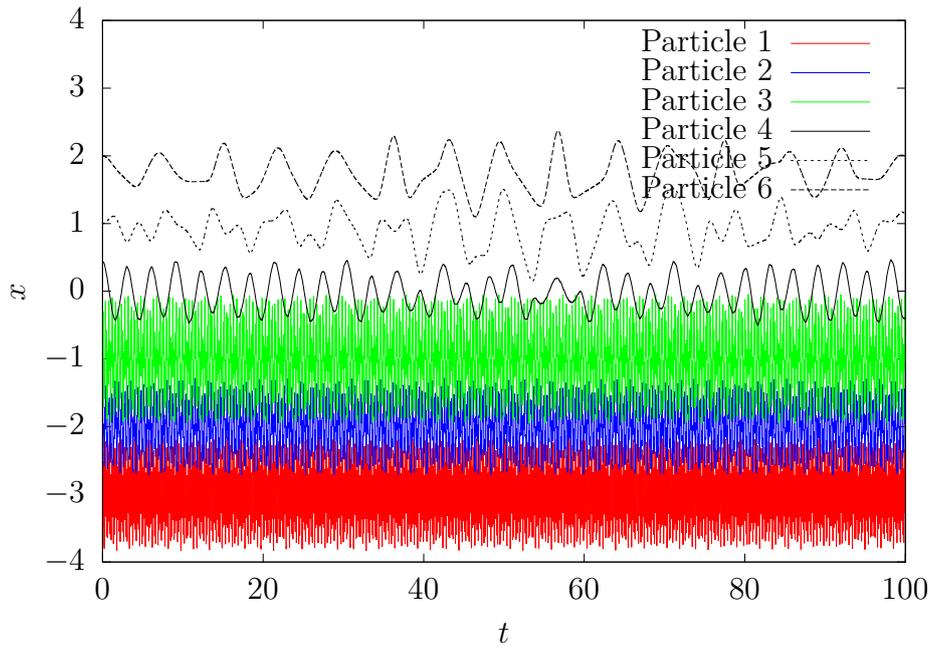}
\caption{Fermi--Pasta--Ulam system with slow-fast dynamics: Position dynamics for the asynchronous scheme ($h_S = 0.01$, $h_F=2\cdot 10^{-4}$)}
\label{fig:Position_slow_fast}
\end{center}
\end{figure}

\begin{figure}[!htp]
\begin{center}
\input{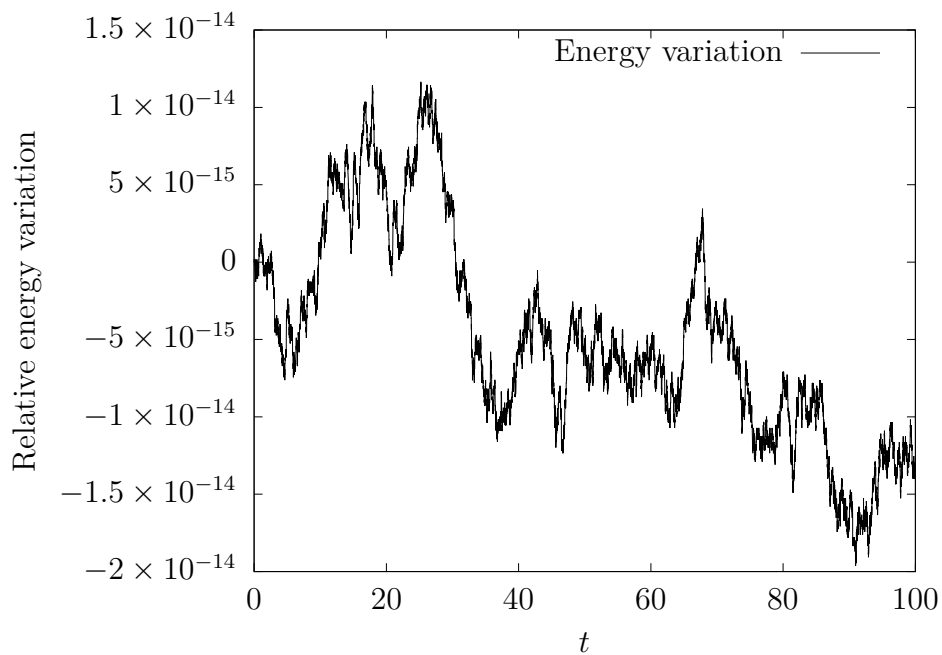}
\caption{Fermi--Pasta--Ulam system with slow-fast dynamics: Relative variation of the discrete pseudo-energy $\tilde{H}^n$ for the asynchronous scheme ($h_S = 0.01$, $h_F=2\cdot 10^{-4}$)}
\label{fig:Energy_slow_fast}
\end{center}
\end{figure}

The computational cost of the scheme is proportional to the number
$\mathcal{N}$ of force evaluations. With an $n$-point 
quadrature and a total integration time $T$, the numbers of force
evaluations $\mathcal{N}_s$ and $\mathcal{N}_a$ for the synchronous
and asynchronous schemes respectively on the present slow-fast
problem are given by:
\[\mathcal{N}_s = (n-1)T\frac{2m+1}{h_F},\qquad \mathcal{N}_a =
(n-1)T\left(\frac{m+1}{h_F}+\frac{m}{h_S}\right).\]
Recalling that $K=\frac{h_S}{h_{F}}\geq 1$ is the number of fast steps per slow
step, the cost reduction $\eta$ of the asynchronous scheme with
respect to the synchronous scheme is given by
\[
\eta = \frac{\mathcal{N}_a}{\mathcal{N}_s} = \frac{1+\frac{m}{(m+1)K}}{1+\frac{m}{m+1}}.
\]
For $h_S=0.01$, $h_F=2\cdot 10^{-4}$ and $T = 100$, $\mathcal{N}_a=1.015\cdot
10^8$, to be compared with $\mathcal{N}_s=1.75\cdot 10^8$. As $m$
increases,
\[
\eta \xrightarrow{m\to+\infty} \frac{1+\frac{1}{K}}{2}.
\]
When the number of fast subiterations $K$ increases, $\eta$ tends to
$0.5$, which means that the computational cost reduction of the
asynchronous scheme compared to the synchronous scheme approaches
$50\%$. This is the best-case scenario, since the computational cost
is concentrated on the fast dynamics where frequent evaluations are
required, whereas the slow dynamics is almost costless.

In order to assess the accuracy of the asynchronous scheme, we
consider the $L^\infty$-error of the position of the asynchronous
solution with respect to the synchronous solution using the small
time-step $h_F$. Figure \ref{fig:conv_dt_lent} (left panel) 
shows the evolution of
the error as the coarse time-step $h_S$ is refined, with fixed
fine time-step $h_F=10^{-4}$. We observe a second-order convergence of the
error. Figure \ref{fig:conv_dt_fast} (right panel) displays the evolution of the
error as the fine time-step $h_F$ is further refined, 
with fixed coarse time-step
$h_S=10^{-2}$. We observe that the error decreases until it reaches a
plateau, which is due to the error on the slow particles. These
observations confirm that reducing the fine time-step beyond $h_F =
h_S/50$ does not significantly improve the error since the error is
dominated by the error on the slow particles. Conversely, the error
reduction due to the coarse time-step reduction is not compromised by
the asynchronous scheme.

\begin{figure}[!htp]
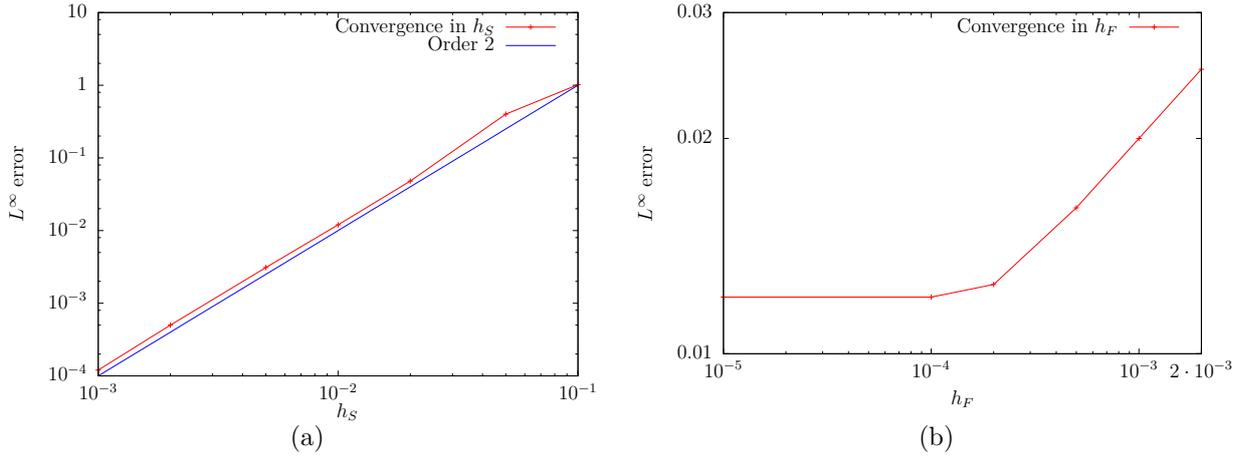

\begin{center}
\subfloat[]{\label{fig:conv_dt_lent}
\resizebox{0.5\textwidth}{!}{\input{conv_dt_lent}}
}
\subfloat[]{\label{fig:conv_dt_fast}
\resizebox{0.5\textwidth}{!}{\input{conv_dt_fast}}
}
\caption{Fermi--Pasta--Ulam system with slow-fast dynamics:
Convergence of the asynchronous scheme \protect\subref{fig:conv_dt_lent} with
respect to the coarse time-step $h_S$, with fixed fine time-step
$h_F=10^{-4}$, and \protect\subref{fig:conv_dt_fast} with
respect to the fine time-step $h_F$, with fixed coarse time-step $h_S=10^{-2}$} 
\end{center}
\end{figure}

Finally, a convergence test is carried out with a constant ratio $\frac{h_S}{h_F} = 25$ and using the five-point Gauss--Lobatto quadrature of order 7. The error is measured as previously by the $L^{\infty}$-error on the positions between the synchronous and asynchronous schemes. The results are presented in Figure \ref{convergence asynchrone}. The synchronous method is used with a constant time-step $h_F$. We observe second-order convergence as both time-steps are refined simultaneously.

\begin{figure}[!htp]
\begin{center}
\includegraphics[scale=0.7]{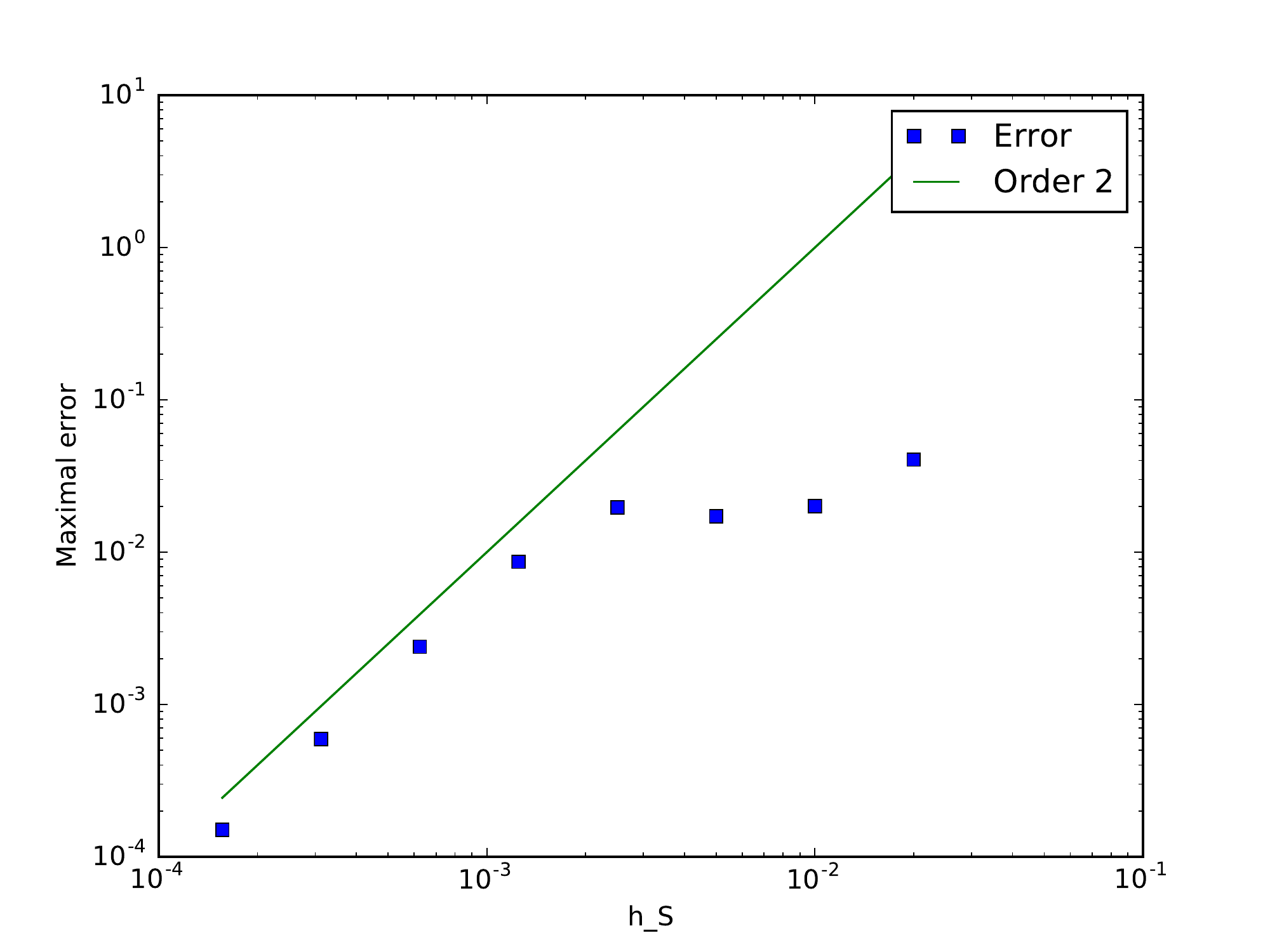}
\caption{Fermi--Pasta--Ulam system with slow-fast dynamics: $L^{\infty}$-error on the positions between the synchronous scheme (with time-step $h_F$) and the asynchronous scheme (with time-steps $h_S$ and $h_F$ having fixed ratio $\frac{h_S}{h_F} = 25$).}
\label{convergence asynchrone}
\end{center}
\end{figure}

\subsubsection{Inhomogeneous wave propagation}

As a physically relevant example of the slow-fast test case, we
consider the propagation of a wave in a linear elastic material in
dimension $d=1$, with an inhomogeneous speed of sound. Denote the
domain $\Omega$, $u^0:\Omega\to\R$ and $v^0:\Omega\to\R$ initial
conditions for displacement and velocity respectively, and
$u: \Omega \times \R^+ \to\R$ the displacement, $u$ follows the equations:
\begin{equation}\label{eq:inhom_wave}
\left\{
\begin{aligned}
&\partial_{tt}u = \partial_x \Big( c(x)^2 \partial_{x}u\Big) \quad  \text{ in } \Omega,\\
&u_{|\partial\Omega} = 0, \\ %, \quad \partial_t u_{|\partial\Omega} = 0,\\
&u(x,0) = u^0(x), \quad \partial_tu(x,0) = v^0(x).
\end{aligned}
\right.
\end{equation}
We take $\Omega=(0,1)$ and we set
\[
c(x) = \begin{cases}
        10 &\text{ if }x\leq 0.5,\\
        1 &\text{ if } x>0.5.
     \end{cases}
\]
Setting $N\in\mathbb{N}^*$, $\Delta x = \frac{1}{N}$ and $x_i=i\Delta x$ for all
$i\in\{0,\dots,N\}$, the partial differential equation
(\ref{eq:inhom_wave}) can be semi-discretized in space with the following centered finite
difference scheme (which is equivalent to a discretization using $H^1$-conforming $\mathbb{P}_1$ Lagrange finite elements after lumping the mass matrix):
\begin{equation}\label{eq:discrete_wave}
\left\{
\begin{aligned}
&\frac{d^2u_i}{dt^2} = \frac{1}{\Delta x^{2}} \left(c(x_{i-1/2})^2(u_{i-1}-u_i)-c(x_{i+1/2})^2(u_i-u_{i+1})\right) \quad \text{ for }
i\in\{1,\dots,N-1\},\\
&u_0 = u_N = 0,\quad \frac{du_0}{dt} = \frac{du_N}{dt} = 0,\\
&u_i(0) = u^0(x_i), \quad \frac{du_i}{dt}(0) = v^0(x_i).
\end{aligned}
\right.
\end{equation}
Setting $\bm{q} = (u_i)_{0\leq i\leq N}$, $\bm{p}
= \left(\frac{du_i}{dt}\right)_{0\leq i\leq N}$ and  $\omega_{i-1/2}
= \frac{c(x_{i-1/2})}{\Delta x}$, the ordinary differential equation
in (\ref{eq:discrete_wave}) is derived from the following Hamiltonian:
\[
H(\bm{p},\bm{q}) = \frac{1}{2} \sum_{i=1}^{N-1}{p^2_{i}}
+ \frac{1}{2}\sum_{i=1}^{N}{\omega_{i-1/2}^2 (q_{i} - q_{i-1})^2}.
\]
The CFL condition (\ref{eq:abs_stab}) becomes 
\[
h < 2\frac{\Delta x}{\omega_{i-1/2}}, \qquad \forall
i \in \{1,\dots,N\}.\] For the indices $i$ such that $x_i\leq 0.5$,
one must then take $h\leq 0.2\Delta x$, while for $x_i>0.5$, it
suffices that $h\leq 2\Delta x$. In what follows, we therefore set the
slow (resp. fast) particules as the elements $i$ such that $x_i>0.5$
(resp. $x_i< 0.5$) and define $h_S = \Delta x$ and $h_F = 0.1\Delta
x$. The mixed particle is the particle at the interface between
the fast and slow particles.

The numerical solution and the exact solution for the displacement and the
velocity computed with $\Delta x=5\times 10^{-4}$ are presented in
Figures \ref{fig:inhom_wave_disp} and \ref{fig:inhom_wave_vit} respectively.
The system is initialized with the functions
\[
u^0(x)= 10^{-2}e^{-(20(x-0.2))^2}\mathds{1}_{(0,0.5)}(x), \quad v^0(x)= 8(x-0.2)e^{-(20(x-0.2))^2}\mathds{1}_{(0,0.5)}(x).
\]
The initial condition propagates to the right with the speed of sound
$c_1=10$, until it reaches $x=0.5$. At the boundary between the slow
and fast domain, it is partly transmitted to the right with speed of
sound $c_2=1$ and partly reflected with speed $-c_1$. The
reflected wave reflects again on the left boundary $x=0$ of the
domain. Successive reflections and transmissions occur, which result
in the final state in Figures \ref{subfig:f} and \ref{subfig2:f}. The
exact solution can be expressed as follows for all $t>0$:
\[
\begin{aligned}
\forall x\in(0,0.5),\quad u(x,t) &= \sum_{k\geq
0}{\left(\frac{c_2-c_1}{c_1+c_2}\right)^k\left(
u_0(x+k-c_1t)-u_0(k-x-c_1t)\right)},\\
\forall x\in(0.5,1),\quad u(x,t) &= \frac{2c_1}{c_1+c_2}\sum_{k\geq
0}{\left(\frac{c_2-c_1}{c_1+c_2}\right)^k
u_0\left(\frac{c_1}{c_2}(x-0.5)+k+0.5-c_1 t\right)}.
\end{aligned}
\]

The numerical solution matches very well the exact solution. We can
observe slight overshoots near the extrema and at the tail of the
peaks, especially in the slow domain. This can be explained by the
fact that the space-discretization (\ref{eq:discrete_wave}) is slightly
dispersive so that steep variations tend to generate oscillations
(similar to the Gibbs phenomenon). 
Figure \ref{fig:error_Nevals}
presents the behavior of the error for the asynchronous and
synchronous schemes with respect to the number of force
evaluations. We consider the maximal error at the end of the simulation, i.e., at $t=0.5$. The results are obtained by letting $h_S,h_F\to 0$ while keeping the ratio $h_S/h_F$ fixed. The convergence is of order 1 with respect to the number of force evaluations. Since the number of force evaluations scales like $(\Delta x \cdot \Delta t)^{-1}$, this is compatible with a second-order convergence in space and in time after taking into account the CFL condition.
In conclusion, the asynchronous scheme displays similar errors to the
synchronous scheme, with roughly half of the number of force evaluations
involved as noted in Section \ref{sec:slow_fast}. This confirms the
efficiency of the asynchronous scheme.

\begin{figure}[!htp]
\begin{center}
\begin{minipage}{0.49\textwidth}
\subfloat[$t=0$]{
\resizebox{\textwidth}{!}{
\input{deplacement_100_0.tex}
}
\label{subfig:a}
}\\
\subfloat[$t=0.2$]{
\resizebox{\textwidth}{!}{
\input{deplacement_100_2.tex}
}
\label{subfig:c}
}\\
\subfloat[$t=0.4$]{
\resizebox{\textwidth}{!}{
\input{deplacement_100_4.tex}
}
\label{subfig:e}
}
\end{minipage}
\begin{minipage}{0.49\textwidth}
\subfloat[$t=0.1$]{
\resizebox{\textwidth}{!}{
\input{deplacement_100_1.tex}
}
\label{subfig:b}
}\\
\subfloat[$t=0.3$]{
\resizebox{\textwidth}{!}{
\input{deplacement_100_3.tex}
}
\label{subfig:d}
}\\
\subfloat[$t=0.5$]{
\resizebox{\textwidth}{!}{
\input{deplacement_100_5.tex}
}
\label{subfig:f}
}
\end{minipage}
\end{center}
\caption{Inhomogeneous wave propagation: Displacement $u$ for $\Delta x =
5\times 10^{-4}$ at
times \protect\subref{subfig:a} $t=0$, \protect\subref{subfig:b}
$t=0.1$, \protect\subref{subfig:c} $t=0.2$, \protect\subref{subfig:d}
$t=0.3$, \protect\subref{subfig:e} $t=0.4$, \protect\subref{subfig:f}
$t=0.5$}
\label{fig:inhom_wave_disp}
\end{figure}

\begin{figure}[!htp]
\begin{center}
\begin{minipage}{0.49\textwidth}
\subfloat[$t=0$]{
\resizebox{\textwidth}{!}{
\input{vitesse_100_0.tex}
}
\label{subfig2:a}
}\\
\subfloat[$t=0.2$]{
\resizebox{\textwidth}{!}{
\input{vitesse_100_2.tex}
}
\label{subfig2:c}
}\\
\subfloat[$t=0.4$]{
\resizebox{\textwidth}{!}{
\input{vitesse_100_4.tex}
}
\label{subfig2:e}
}
\end{minipage}
\begin{minipage}{0.49\textwidth}
\subfloat[$t=0.1$]{
\resizebox{\textwidth}{!}{
\input{vitesse_100_1.tex}
}
\label{subfig2:b}
}\\
\subfloat[$t=0.3$]{
\resizebox{\textwidth}{!}{
\input{vitesse_100_3.tex}
}
\label{subfig2:d}
}\\
\subfloat[$t=0.5$]{
\resizebox{\textwidth}{!}{
\input{vitesse_100_5.tex}
}
\label{subfig2:f}
}
\end{minipage}
\end{center}
\caption{Inhomogeneous wave propagation: Velocity $\frac{du}{dt}$ for $\Delta x =
5\times 10^{-4}$ at
times \protect\subref{subfig2:a} $t=0$, \protect\subref{subfig2:b}
$t=0.1$, \protect\subref{subfig2:c} $t=0.2$, \protect\subref{subfig2:d}
$t=0.3$, \protect\subref{subfig2:e} $t=0.4$, \protect\subref{subfig2:f}
$t=0.5$}
\label{fig:inhom_wave_vit}
\end{figure}

\begin{figure}[!htp]
\input{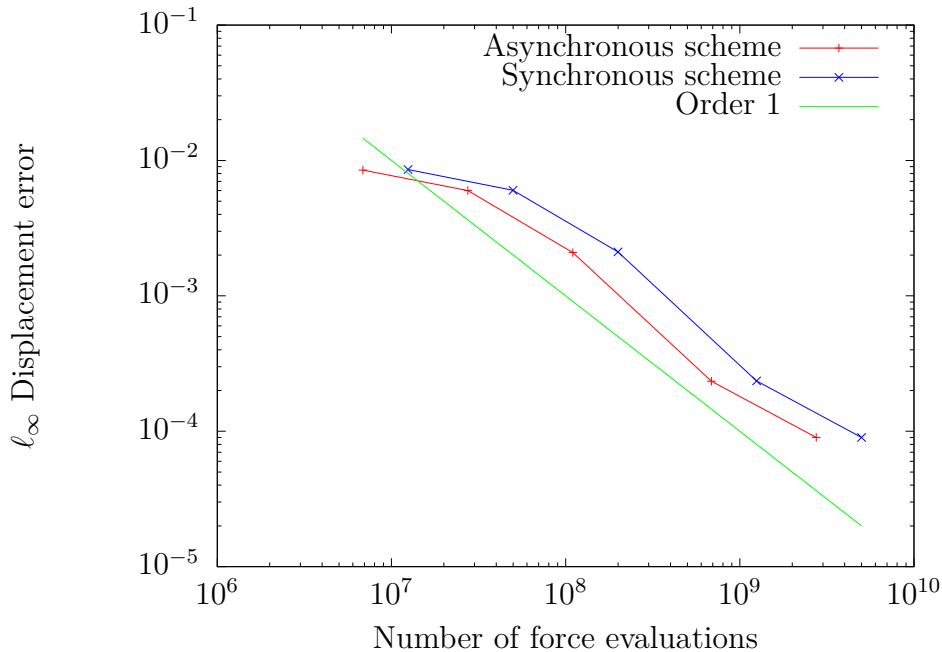}
\caption{Inhomogeneous wave propagation: Maximal displacement error at $t = 0.5$ against the number of
force evaluations for the asynchronous and the synchronous
schemes}\label{fig:error_Nevals}
\end{figure}

\section{Conclusion}

In this paper, a new explicit pseudo-energy conserving time-integration
scheme has been proposed. It is capable of handling general nonlinear
Hamiltonian systems and has been tested on classical numerical
benchmarks and on a nonlinear wave propagation problem. The present
scheme enables the use of local time-stepping
strategies to circumvent stiff CFL conditions on the time-step and
to enhance computational efficiency in the context of slow-fast dynamics.

Various perspectives of the present work can be considered. We believe
that the time-integration of dissipative systems should be a
straightforward extension of the present scheme. Variational
integrators have been proposed for dissipative systems and have proven
to be able to accurately track the physical dissipation of
energy~\cite{kane2000variational}. Other possible developments lie in
the adaptation of the scheme to constrained Hamiltonian
systems~\cite{leyendecker2008variational}, such as mechanical contact
problems~\cite{kane1999finite,wohlmuth2011variationally} and rigid
body
rotations~\cite{krysl2005explicit,salomon2008energy,mariotti2016new}. Another
perspective is the high-order extension of the present scheme.

\subsection*{Acknowledgements}
The authors would like to thank F. Legoll (University Paris-Est, Navier Laboratory) for stimulating discussions on the integration of
Hamiltonian dynamics.
The authors are also thankful to the anonymous referees for their insightful remarks.

\bibliographystyle{plain} %-en
\bibliography{bibliographie} %Article

\end{document}